\documentclass[11pt]{article}
\usepackage[margin=1in]{geometry}
\usepackage{amssymb,amsmath,latexsym}
\usepackage{enumerate}
\usepackage{graphicx}              
\usepackage{amsfonts}              
\usepackage{amsthm}                
\usepackage{url}
\usepackage{marvosym}
\usepackage{color}
\usepackage[backend=bibtex,style=alphabetic,doi=false,isbn=false,url=false]{biblatex}
\usepackage[hidelinks,pdfencoding=auto]{hyperref}
\usepackage{soul}
\usepackage{bm}
\usepackage{tikz-cd}
\usepackage[inline]{enumitem}
\usepackage{tasks}

\newtheorem{theorem}{Theorem}[section]
\newtheorem*{theorem*}{Theorem}
\newtheorem{lemma}[theorem]{Lemma}
\newtheorem{proposition}[theorem]{Proposition}
\newtheorem{corollary}[theorem]{Corollary}

\newtheorem*{main-thm}{Main Theorem}

\theoremstyle{definition}
\newtheorem{definition}[theorem]{Definition}
\newtheorem{example}[theorem]{Example}
\newtheorem{remark}[theorem]{Remark}
\newtheorem{question}[theorem]{Question}

\newtheorem{notation}[theorem]{Notation}

\newcommand{\op}[1]{\operatorname{#1}}

\newcommand{\ZZ}{\mathbb{Z}}
\newcommand{\RR}{\mathbb{R}}
\newcommand{\CC}{\mathbb{C}}
\newcommand{\QQ}{\mathbb{Q}}

\makeatletter
\providecommand*{\twoheadrightarrowfill@}{%
  \arrowfill@\relbar\relbar\twoheadrightarrow
}
\providecommand*{\twoheadleftarrowfill@}{%
  \arrowfill@\twoheadleftarrow\relbar\relbar
}
\providecommand*{\xtwoheadrightarrow}[2][]{%
  \ext@arrow 0579\twoheadrightarrowfill@{#1}{#2}%
}
\providecommand*{\xtwoheadleftarrow}[2][]{%
  \ext@arrow 5097\twoheadleftarrowfill@{#1}{#2}%
}
\makeatother

\begin{document}

\title{On enumerating factorizations in reflection groups}
\author{\Large Theo Douvropoulos\thanks{This work was supported by the European Research Council, grant ERC-2016- STG 716083 ``CombiTop''.}}

\newcommand{\Address}{{
  \bigskip
  \footnotesize
	
  Theo~Douvropoulos, \textsc{IRIF, UMR CNRS 8243, Universit\'e Paris Diderot, Paris 7, France}\par\nopagebreak
  \textit{E-mail address:} \texttt{douvr001@irif.fr}	

}}

\date{}
\maketitle
\begin{abstract}
We describe an approach, via Malle's permutation $\Psi$ on the set of irreducible characters $\op{Irr}(W)$, that gives a uniform derivation of the Chapuy-Stump formula for the enumeration of reflection factorizations of the Coxeter element. It also recovers its weighted generalization by delMas, Reiner, and Hameister, and further produces structural results for factorization formulas of arbitrary regular elements. 
\end{abstract}

\section{Introduction}

A famous theorem of Cayley states that there are $n^{n-2}$ vertex-labeled trees on $n$ vertices. The same number,\footnote{The two objects are naturally related via a satisfying overcounting argument due to D{\'e}nes \cite{denes-first-paper}.} as Hurwitz knew \cite{hurwitz-1891-paper} already by the end of the $19^{\text{th}}$ century, enumerates the set of shortest length factorizations $t_1\cdots t_{n-1}=(12\cdots n)\in S_n$ of the long cycle into transpositions $t_i$. A natural generalization of this problem, that Hurwitz himself had later considered \cite{hurwitz-1901-paper}, is to enumerate such factorizations of arbitrary length.

It took almost a hundred years for the community to return to this question, but by the end of the 80's Jackson \cite[Corol.~4.2]{jackson-first-proof-of-arb-length} had computed an explicit answer. If $\op{FAC}_{S_n}(t)$ denotes the exponential generating function for the number of arbitrary length factorizations of the long cycle in transpositions (see \eqref{Eq: Frob formula no Cox}), then Jackson's result can be reinterpreted as follows: \begin{equation}
\op{FAC}_{S_n}(t)=\dfrac{e^{t\binom{n}{2}}}{n!}\big( 1-e^{-tn}\big)^{n-1}.\label{EQ: Jackson's formula}
\end{equation}

As it often happens with some of the most fascinating properties of the symmetric group, the previous statements are special cases of more general theorems that hold for all reflection groups $W$. A natural analog of the long cycle is the Coxeter element $c\in W$, while transpositions are replaced by reflections. Then, if $W$ is of rank $n$, $\mathcal{R}$ denotes its set of reflections, and $h$ is the order of $c$, Bessis \cite[Prop.~7.6]{bes-kpi1} proved the following enumeration: \begin{equation}
\#\Big\{ (t_1,\cdots,t_n)\in \mathcal{R}^n\ |\ t_1\cdots t_n=c\Big\}\ =\ \dfrac{h^nn!}{|W|}.\label{EQ: Hurwitz number for W}
\end{equation}

The $W$-analog of Jackson's formula \eqref{EQ: Jackson's formula} regarding arbitrary length factorizations was discovered (and proved) by Chapuy and Stump \cite{chapuy-stump} soon after. If $\op{FAC}_W(t)$ denotes the corresponding exponential generating function, they showed that \begin{equation}
\op{FAC}_W(t)=\dfrac{e^{t|\mathcal{R}|}}{|W|}\big( 1-e^{-th}\big)^n.\label{EQ: Chapuy-Stump formula}
\end{equation}

The reduced case \eqref{EQ: Hurwitz number for W}, which can easily be derived by calculating the leading term of $\op{FAC}_W(t)$, has a long history and appears in connection to many a mathematical endeavour. It originated in singularity theory \cites[Conj.~(3.5)]{looij-complement-bifurc}{deligne-letter-to-looijenga}, in combinatorics it appeared as the number of maximal chains in the noncrossing lattice $NC(W)$ \cite[Prop.~9]{chapoton-enumerative-properties}, and more importantly it was essential in Bessis' proof of the $K(\pi,1)$-conjecture \cite{bes-kpi1} (see \cite[\S~1]{theo-thesis} for a detailed presentation).

\subsubsection*{A uniform argument}

Neither \eqref{EQ: Chapuy-Stump formula} nor \eqref{EQ: Hurwitz number for W} are well understood. Although the statements are uniform for all well-generated groups, the proofs of Bessis and Chapuy-Stump have relied on the Shephard-Todd classification (a common misfortune for theorems regarding reflection groups). As it happens, the main goal of this paper is to provide a case-free explanation for these formulas. 

The standard approach towards results like \eqref{EQ: Jackson's formula} and \eqref{EQ: Chapuy-Stump formula} is via the Frobenius lemma (Thm.~\ref{Thm: Frobenius lemma}), which involves summing over all irreducible characters of a group $W$. For that matter, one of the main obstacles to producing a conceptual proof for \eqref{EQ: Chapuy-Stump formula} lies in that we have no nice, uniform construction of irreducible characters for complex reflection groups. Only for Weyl groups there is Springer's correspondence \cite{springer-correspondence}, which is however technically difficult for computations. 

In this work we also start with the Frobenius Lemma, but instead of explicitly computing the characters $\chi\in\op{Irr}(W)$, we group them together with respect to an invariant called the Coxeter number $c_{\chi}$ (see Defn.~\ref{Defn: Coxeter numbers}). Then, Malle's cyclic action $\Psi$ on $\op{Irr}(W)$ allows us to cancel the contribution of those $\chi$ for which $c_{\chi}$ is not a multiple of $h$. The resulting expression is very rigid (Thm.~\ref{Thm: Main, not weighted}) and the mere knowledge of bounds for the $c_{\chi}$ allows us to complete the proof.

Ours is not the first approach towards a uniform proof of \eqref{EQ: Chapuy-Stump formula}. In \cite{michel-deligne-lusztig-derivation}, Michel also considers a grouping of the characters; the partition given by Lusztig's families. This is finer (and much more technologically advanced) and although the argument gives a very satisfying connection between \eqref{EQ: Jackson's formula} and \eqref{EQ: Chapuy-Stump formula}, it requires the existence of the elusive ``spets" \cite{broue-malle-michel-split-spetses} when $W$ is not a Weyl group.

Moreover, our strategy applies in further generality and produces structural results for any regular element $g\in W$ (which become explicit formulas for a larger class of groups than the well-generated ones, see Corol.~\ref{Cor: Chapuy-Stump for dn}). In addition, a refined version in Section~\ref{Section: The weighted enumeration} recovers (uniformly) and extends the main result of \cite{reiner-delMas-Hameister} on a weighted version of the Chapuy-Stump formula \eqref{EQ: Chapuy-Stump formula}.

When $W$ is a real reflection group, all our theorems are completely case-free. In the complex case, although our approach is indeed uniform, it relies on the BMR-freeness theorem, a property of the Hecke algebra $\mathcal{H}(W)$ that is currently proven in a case-by-case way (see \S\ref{Section: On the uniformity of the proofs} for details).

\vspace{-0.1cm}
\subsubsection*{Summary}

The main results of this paper (Thm.~\ref{Thm: Main, not weighted} and Thm.~\ref{Thm: Main, weighted}) are presented in Sections~\ref{Section: Frobenius lemma via Coxeter numbers} and \ref{Section: The weighted enumeration} which can be read essentially independently of the rest. They rely on a key technical lemma (Prop.~\ref{Prop: The key technical lemma}) that describes how Malle's permutation $\Psi$ (Defn.~\ref{Defn: Malle's character permutations}) affects character values on regular elements. The material in Sections~\ref{Section: Complex reflection groups and regular elements} and \ref{Section: Hecke algebras and the technical lemma} essentially builds up to the proof of that lemma.

In particular, the two theorems are valid for all regular elements due to a characterization of the latter ones as those that have lifts in the braid group that are roots of the full twist (see Prop.~\ref{Prop: regular elements lift to roots of the full twist}). For this reason, we have reviewed in some detail in \S\,\ref{Section: Complex reflection groups and regular elements} the various statements about the topological definition of the braid group and its abelianization, the full twist and the lifts of regular elements.

Similarly in Section~\ref{Section: Hecke algebras and the technical lemma}, building towards the technical lemma, we recall the definition of the Hecke algebras given at \cite{broue-malle-rouquier-braid-hecke}, and reproduce some key character calculations from \cite{broue-michel-french-paper}. The reader who is comfortable with these concepts might skip the bulk of these sections, but we hope the presentation will prove sufficient for those unfamiliar with Hecke algebras, but who might want to further pursue their combinatorial consequences.

\section{Complex reflection groups and regular elements}
\label{Section: Complex reflection groups and regular elements}

Given a complex vector space $V\cong \CC^n$, we call a {\it finite} subgroup $W\leq \op{GL}(V)$ a {\it\color{blue} complex reflection group} if it is generated by unitary reflections. These are $\CC$-linear maps $t$ whose fixed spaces $V^t:=\op{ker}(t-\op{id})$ are hyperplanes (i.e. $\op{codim}(V^t)=1$). We further say that $W$ is {\it irreducible} if it has no stable linear subspaces apart from $V$ and $\{0\}$. Shephard and Todd \cite{shephard-and-todd} classified irreducible complex reflection groups into an infinite $3$-parameter family $G(r,p,n)$ and 34 exceptional cases indexed $G_4$ to $G_{37}$. The reader may consult the classical references \cites{kane-book-reflection-groups}{broue-book-braid-groups}{lehrer-taylor-book} for the material in this section.

We denote by $\mathcal{R}$ the set of reflections of $W$ and we write $\mathcal{A}$ for the associated arrangement of fixed hyperplanes. For such a hyperplane $H$, let $W_H$ be its pointwise stabilizer. It consists of the identity and the reflections that fix $H$. Furthermore, because unitary reflections are semisimple, $W_H$ is cyclic.

Now, if $e_H:=|W_H|$ is the size of this cyclic group and $t_H$ is one of its generators, the set of reflections $\mathcal{R}$ can be partitioned as:\begin{equation} \mathcal{R}=\bigcup_{H\in\mathcal{A}}\{ t_H,\cdots, t_H^{e_H-1} \}.\label{EQ: Reflection set partition into cyclis}\end{equation}
The reflection group $W$ acts on $\mathcal{A}$ determining orbits of hyperplanes which we will denote by $\mathcal{C}\in\mathcal{A}/W$. The size $\omega_{\mathcal{C}}$ of an orbit $\mathcal{C}$ is given by $\omega_{\mathcal{C}}:=[W:N_W(H)]$ (for any $H\in \mathcal{C}$). All elements $H\in\mathcal{C}$ have conjugate stabilizers $W_H$ and we write $e_{\mathcal{C}}$ for their common order.

With this notation, the cardinalities of the set of reflections $\mathcal{R}$ and of the set of reflecting hyperplanes $\mathcal{A}$ are given by $$|\mathcal{R}|=\sum_{\mathcal{C}\in\mathcal{A}/W}\omega_{\mathcal{C}}(e_{\mathcal{C}}-1)\quad\quad\text{ and }\quad\quad |\mathcal{A}|=\sum_{\mathcal{C}\in\mathcal{A}/W}\omega_{\mathcal{C}}.$$
Notice that if some $e_{\mathcal{C}}\neq 2$, then $|\mathcal{R}|$ and $|\mathcal{A}|$ are not equal.
\subsection{Braid groups and braid reflections}
\label{Section: Braid groups and generators of the monodromy}

We say that a vector $v\in V$ is {\it regular} if it is not contained in any reflection hyperplane and we write $V^{\op{reg}}:=V\setminus \mathcal{A}$ for the set of regular vectors. We define the {\it pure braid group} $P(W):=\pi_1(V^{\op{reg}})$ to be the fundamental group of the regular space $V^{\op{reg}}$. It is a theorem of Steinberg that the action of $W$ on $V$ is free precisely on $V^{\op{reg}}$. 

Steinberg's theorem implies that the restriction of the quotient map $\rho: V\rightarrow V/W$ on $V^{\op{reg}}$ is a Galois covering. We define the {\it \color{blue} braid group} $B(W):=\pi_1(V^{\op{reg}}/W)$ to be the fundamental group of the base of this covering and use the following short exact sequence \cite[Prop.~1.40]{hatcher-book-alg-top} to obtain a surjection $\pi:B(W)\twoheadrightarrow W$:

\begin{equation}\label{1PBW1}
1\rightarrow\underset{\overset{\rotatebox{90}{\,:=}}{\displaystyle P(W)}}{\pi_1(V^{\op{reg}})}\lhook\joinrel\xrightarrow{\ \displaystyle \rho_*\ }
\underset{\overset{\rotatebox{90}{\,:=}}{\displaystyle B(W)}}{\pi_1(V^{\op{reg}}/W)}
\xtwoheadrightarrow{\ \displaystyle \pi} W\rightarrow 1.
\end{equation}

Given a choice of a basepoint $x_0\in V^{\op{reg}}$, a loop $\bm b\in B(W)$ lifts to a path that connects $x_0$ to $\bm b_*(x_0)$ (we call this the {\it Galois action} of $b$). Then, we define $w:=\pi(\bm b)$ to be the {\it unique} element $w\in W$ such that $w\cdot x_0=b_*(x_0)$. The significance of \eqref{1PBW1} lies in that it gives a topological interpretation of $W$ as the group of deck transformations of the covering map $\rho:V^{\op{reg}}\rightarrow V^{\op{reg}}/W$.

A reflection group $W$ acts on the polynomial algebra $\CC[V]:=\op{Sym}(V^*)$ of the  space $V$ by precomposition (i.e. $w*f(\bm v):=f(w^{-1}\cdot \bm v)$). The Shephard-Todd-Chevalley theorem \cites{shephard-and-todd}{chevalley-improves-Shep-and-Todd} states then that the algebra of invariant polynomials $\CC[V]^W:=\{f\in\CC[V]:\ w*f=f\ \forall w\in W\}$ is itself a polynomial algebra. We choose homogeneous generators for it, which we denote by $f_i$ and order them by increasing degree $\op{deg}(f_i)=:d_i$. The numbers $d_i$ are independent of the choise of the $f_i$'s and are called the {\it \color{blue} fundamental degrees} of $W$.

In this setting, we can further understand the quotient morphism $\rho:V\rightarrow V/W$ by studying its algebro-geometric structure. In particular (and this holds for any finite subgroup of $\op{GL}(V)$) the map $\rho$ is a finite morphism and the quotient $V/W$ can be realized as the affine variety $\op{Spec}\big(\CC[V]^W\big)$ \cite[see][Exer.~13.2-4 and Sec.~1.7]{eisenbud-commutative-algebra}. The Shephard-Todd-Chevalley theorem states then that for reflection groups $W$, the quotient $V/W$ is itself an affine space, so that we may write: 
\begin{equation}\label{quotient map}
\CC^n\cong V\ni{\bf x}:=(x_1,\cdots,x_n)\xrightarrow{\ \displaystyle\rho\ } {\bf f}({\bf x}):=\big(f_1({\bf x}),\cdots,f_n({\bf x})\big)\in W\backslash V\cong \CC^n
\end{equation}

Now the hyperplane arrangement $\mathcal{A}$ (which is the zero set of a collection of linear forms) is an affine variety, stable under the action of $W$. Another consequence of the above is then that its image $\mathcal{H}:=\rho(\mathcal{A})\subset V/W$ is itself a variety; we call it the {\it discriminant hypersurface} of $W$. The braid group becomes thus the fundamental group of a hypersurface complement $B(W)=\pi_1(V/W-\mathcal{H})$.

Such groups have a special set of generators called {\it generators of the monodromy} \cite[Appendix~1]{broue-malle-rouquier-braid-hecke}. These are loops that descend from the basepoint following a path $\gamma$, approach a smooth point of an irreducible component of the hypersurface and make a counterclockwise\footnote{Near a smooth point, an irreducible codimension 1 divisor in $\CC^n$ looks like a line in $\RR^3$; there is a well-defined way to go around it.} loop around it, and finally return following the same path $\gamma$ backwards. 

In our case, the irreducible components of $\mathcal{H}$ are the images $\rho(\mathcal{C})$ of the hyperplane orbits $\mathcal{C}\in\mathcal{A}/W$ (again a consequence of the discussion before \eqref{quotient map}). We will therefore denote the generators of the monodromy for $B(W)$ by $\bm s_{\mathcal{C},\gamma}$. They map (via \eqref{1PBW1}) to a subset of reflections $s_H\in W$ which have determinant $\zeta_{e_{\mathcal{C}}}:=\op{exp}(2\pi i/e_{\mathcal{C}})$ and are called {\it distinguished reflections}. In fact, for this reason, we follow the terminology suggested by Brou{\'e}, Malle, Rouquier, Michel, and Bessis, (see for instance \cite[Defn.~1.6]{bes-kpi1}):

\begin{definition}
The generators of the monodromy of $B(W)$ are called {\it\color{blue} braid reflections}.
\end{definition}

\noindent The powers $\bm s_{\mathcal{C},\gamma}^{e_{\mathcal{C}}}$ are generators of the monodromy for the pure braid group: 

\begin{proposition}\cite[Prop.~2.18]{broue-malle-rouquier-braid-hecke}\label{Prop: (braid reflections)^e_c generate P(W)}
After a choice of basepoint $v\in V^{\op{reg}}$, we can lift the $\bm s_{\mathcal{C},\gamma}$ to paths in $P(W)$. Then the pure braid group $P(W)$ is generated by $\langle \bm s_{\mathcal{C},\gamma}^{e_{\mathcal{C}}}\rangle$ (for all $\mathcal{C},\gamma$) and we have $$W\cong B(W)/\langle \bm s_{\mathcal{C},\gamma}^{e_{\mathcal{C}}}\rangle,$$ where the isomorphism is the same as the one induced by the choice of $v$ via \eqref{1PBW1}.
\end{proposition}

\subsection{The full twist and the abelianization of $B(W)$}

Brou{\'e}-Malle-Rouquier considered \cite[Notation~2.3]{broue-malle-rouquier-braid-hecke} a particular element of the pure braid group $P(W)$; it is fundamental in what follows and for the results in Sections \ref{Section: Frobenius lemma via Coxeter numbers} and \ref{Section: The weighted enumeration}. For an arbitrary regular vector $v\in V^{\op{reg}}$, we define $\bm \pi_v\in\pi_1(V^{\op{reg}},v)$ as the loop given by: \begin{equation}[0,1]\ni t\rightarrow e^{2\pi i t}\cdot v.\label{EQ: full twist definition} \end{equation}

If $\gamma\subset V^{\op{reg}}$ is any path between points $v, v' \in V^{\op{reg}}$, then the cylinder (or torus if $\gamma$ is a loop) $S^1\cdot \gamma$ lies completely inside $V^{\op{reg}}$. This is because $V^{\op{reg}}$, the complement of a central hyperplane arrangement, is stable under multiplication by $\CC^{\times}\supset S^1$. It is immediate from this that:

\begin{lemma}\label{Lemma: full twist is central in fund. groupoid of V^reg}\cite[Lemma~2.5]{broue-malle-rouquier-braid-hecke}\newline
For $v,v',$ and $\gamma$ as above, the loops $\gamma^{-1}\cdot\bm\pi_{v'}\cdot\gamma$ and $\bm\pi_v$ in $P(W,v)$ are homotopic. 
\end{lemma}

This in particular implies that $\bm\pi_v$ is always central in $P(W,v)$. Furthermore, if $v$ and $v'$ have the same image in $V^{\op{reg}}/W$, and since $\rho$ is quasihomogeneous \eqref{quotient map}, the loops $\rho_*(\bm\pi_v)$ and $\rho_*(\bm\pi_{v'})$ are identical. Now, this along with the previous lemma immediately gives:

\begin{corollary}\label{Corol. full twist is central in B(W)}\cite[from Lemma~2.22:~(2)]{broue-malle-rouquier-braid-hecke}\newline
For any regular vector $v\in V^{\op{reg}}$, the element $\rho_*(\bm\pi_v)\in B(W,\rho(v))$ is central.
\end{corollary}

For any two basepoints $v$ and $v'$ of $V^{\op{reg}}$ and a path $\gamma$ between them, there are canonical isomorphisms between the fundamental groups $P(W,v)$ and $P(W,v')$, and between $B(W,\rho(v))$ and $B(W,\rho(v'))$. Since $\bm\pi_v$ and $\rho_*(\bm \pi_v)$ are central, their images will also be central and moreover independent of the path $\gamma$ (in fact, the previous lemma shows that they will be homotopic to $\bm\pi_{v'}$ and $\rho_*(\bm\pi_{v'})$ respectively). We therefore drop the basepoint from the notation, and for convenience we use the same symbol for the image in $B(W)$ as well:
\begin{definition}\cite[Defn.~6.12]{bes-kpi1}\label{Defn: The full twist}\newline
We call this element $\bm\pi$ defined in \eqref{EQ: full twist definition} the {\it\color{blue} full twist}. It is central in $B(W)$ and lies in $P(W)$.
\end{definition}

Brou{\'e}-Malle-Rouquier also consider \cite[Defn.~2.15]{broue-malle-rouquier-braid-hecke} length functions $l_{\mathcal{C}}:B(W)\rightarrow \ZZ$, given as periods of the differential forms $d\op{Log}(\delta_{\mathcal{C}})$ associated to discriminant polynomials $\delta_{\mathcal{C}}$ that cut out the strata $\mathcal{C}$ of $\mathcal{H}$ \cite[Defn.~2.15]{broue-malle-rouquier-braid-hecke}. For a loop $\bm g\in B(W)$, they essentially record how many radians any of its lifts $\bm g'\in P(W)$ wraps around each hyperplane in the orbit $\mathcal{C}\in\mathcal{A}/W$, and weigh the result by $e_{\mathcal{C}}$ (see [ibid,~Thm.~2.17:~Remark]). In particular, they satisfy [ibid,~Prop.~2.16] $$l_{\mathcal{C}}(\bm s_{\mathcal{C'},\gamma})=\delta_{\mathcal{C},\mathcal{C'}},$$ which, since the $\bm s_{\mathcal{C},\gamma}$ generate $B(W)$ (see discussion before Prop.~\ref{Prop: (braid reflections)^e_c generate P(W)}), implies that in fact these length functions completely determine the abelianization $B^{\op{ab}}$ of $B(W)$:

\begin{theorem}\label{Thm: The abelianization of B(W)}\cite[Thm.~2.17:(2)]{broue-malle-rouquier-braid-hecke}\newline
If $\bm s_{\mathcal{C}}^{\op{ab}}$ denotes the image of any $\bm s_{\mathcal{C},\gamma}$ in the abelianization $B^{\op{ab}}$, then $$B^{\op{ab}}=\prod_{\mathcal{C}\in \mathcal{A}/W}\langle \bm s_{\mathcal{C}}^{\op{ab}}\rangle,$$ where each $\langle \bm s_{\mathcal{C}}^{\op{ab}}\rangle$ is infinite cyclic. Moreover, for an element $\bm g\in B(W)$, we have $$ \bm g^{\op{ab}}=\prod_{\mathcal{C}\in\mathcal{A}/W}(\bm s_{\mathcal{C}}^{\op{ab}}\big)^{l_{\mathcal{C}}(\bm g)}.$$ 
\end{theorem}

\noindent By definition the full twist $\bm \pi$ rotates once around each of the $\omega_{\mathcal{C}}$-many hyperplanes in any orbit $\mathcal{C}$: 

\begin{corollary}\label{Cor: the abelianization of the full twist}\cite[Cor.~2.26 and Lemma~2.22:(2)]{broue-malle-rouquier-braid-hecke}\newline
Let $\bm \pi^{\op{ab}}$ be the image in $B^{\op{ab}}$ of the full twist $\bm \pi$. Then we have $$\bm\pi^{\op{ab}}=\prod_{\mathcal{C}\in\mathcal{A}/W}\big(\bm s_{\mathcal{C}}^{\op{ab}}\big)^{e_{\mathcal{C}}\cdot \omega_{\mathcal{C}}}.$$
\end{corollary}

\subsection{Regular elements and roots of the full twist}
\label{Section:Regular elements and roots of the full twist}

Although our initial purpose for this project was to give a uniform proof of the Chapuy-Stump formula \eqref{EQ: Chapuy-Stump formula} which regards Coxeter elements, it soon became clear that the techniques developed (see Lemma~\ref{Lemma: only multiples of |g| contribute}) apply to the larger class of Springer-regular elements. The crucial property these elements share is that they lift to roots of (powers of) the full twist $\bm \pi$ (Defn.~\ref{Defn: The full twist}).

\begin{definition}\cite{springer-regular-elements}\label{Defn: regular element}
Recall the space $V^{\op{reg}}$ of {\it regular} vectors; namely those that do not lie in any hyperplane $H\in\mathcal{A}$. We say that an element $g\in W$ is {\color{blue} $\zeta$-regular} if it has a regular $\zeta$-eigenvector; all $\zeta$-regular elements are conjugate \cite[Corol.~11.25]{lehrer-taylor-book}. The order $d$ of a $\zeta$-regular element $g$ is equal to the order of $\zeta$ [ibid] and is called a {\it \color{blue} regular number}.
\end{definition}

For real reflection groups $W$, the product $c$ of the simple generators (in any order) is called a Coxeter element, after Coxeter who first computed its order $h$ and eigenvalues \cite{coxeter-duke}. At the same paper, Coxeter observed (and Steinberg later \cite{steinberg-finite-reflection-groups} gave a uniform proof of the fact) that $h$ determines the number of hyperplanes $N$ via the equation $nh=2N$, where $n$ is the dimension of the ambient space $V$. Steinberg's work easily implies also that $c$ is an $e^{2\pi i/h}$-regular element.

Building on that, Gordon and Griffeth (but see also the beginning of \S\,\ref{Section: Malle's character permutations and the technical lemma}) define a Coxeter number\footnote{It is not a priori clear that $h$ is an integer; see Corol.~\ref{Corol: local cox num are integers andbounds}.} for all complex reflection groups as $h=(|\mathcal{R}|+|\mathcal{A}|)/n$. Then, for an arbitrary $W$, a {\it \color{blue} Coxeter element} is defined as a $e^{2\pi i/h}$-regular element. It turns out that Coxeter elements exist precisely when $W$ is {\it \color{blue} well-generated}; namely when it is generated by $n$ reflections. 

It is easy to produce lifts $\bm g\in B(W)$ of regular elements $g\in W$. Indeed, let $g$ be a $\zeta$-regular element, with $\zeta=\op{exp}(2\pi i m/d)$, $(m,d)=1$, and let $x_0$ be one of its $\zeta$-eigenvectors. Consider now the path $\bm \pi_{x_0,\zeta}$ in $V^{\op{reg}}$ that connects $x_0$ and $\zeta x_0$ and is defined by \begin{equation}[0,1]\ni t\rightarrow e^{2\pi i tm/d}x_o.\label{EQ: root of power of full twist}\end{equation}
Since $\zeta x_0=g\cdot x_0$, this determines a loop in $V^{\op{reg}}/W$ that would lift the element $g\in W$, if $x_0$ was the basepoint for $P(W)$. We can easily adjust the construction to deal with a basepoint that is not an eigenvector, and comparing \eqref{EQ: full twist definition} and \eqref{EQ: root of power of full twist} gives:

\begin{proposition}\cite[Prop.~5.24]{broue-book-braid-groups}\label{Prop: regular elements lift to roots of the full twist}
Let $\zeta=\op{exp}(2\pi i m/d)$ be a primitie $d^{\text{th}}$ root of unity, and let $g$ be a $\zeta$-regular element of $W$. Then, $g$ has a lift $\bm g\in B(W)$ such that $\bm g^d=\bm \pi ^m$.
\end{proposition}
\begin{proof}
Let $v\in V^{\op{reg}}$ be the basepoint of $P(W)$ and $\gamma$ an arbitrary path in $V^{\op{reg}}$ that connects $v$ with a $\zeta$-eigenvector $x_0$ of $g$. We view $g$ as a deck transformation of the covering $\rho: V^{\op{reg}}\rightarrow V^{\op{reg}}/W$ and consider the path $(g\cdot \gamma^{-1})\cdot \bm \pi_{x_0,\zeta}\cdot \gamma$.  It connects the points $v$ and $g\cdot v$ and hence determines the following element of the braid group $B(W)$: $$\bm g:=\rho(\gamma)^{-1}\cdot \rho(\bm \pi_{x_0,\zeta})\cdot \rho(\gamma).$$
Because $g$ acts on the line $\CC\cdot x_0$ as multiplication by $\zeta$, we can see that the loop $\rho( \bm \pi_{x_0,\zeta})^d$ lifts to the element $\bm \pi_{x_0}^m=\bm \pi_{\zeta^{d-1}\cdot x_0,\zeta}\cdots \bm\pi_{x_0,\zeta}$ (recall the definition of $\bm \pi_{x_0}$ in \eqref{EQ: full twist definition}). This immediately gives 
$$\bm g^d=\rho(\gamma)^{-1}\cdot \rho(\bm \pi_{x_0})^m \cdot \rho(\gamma),$$ which after the discussion before Defn.~\ref{Defn: The full twist} completes the proof.
\end{proof}

\begin{remark}
The converse of the previous theorem is still true; that is, $d^{\text{th}}$ roots of the full twist exist precisely when $d$ is a regular number \cite[Thm.~12.4]{bes-kpi1}. Moreover, as with Springer-regular elements, Bessis has shown [ibid] that all $d^{th}$ roots of $\bm\pi^m$ are conjugate. Such results essentially ``lift" Springer theory to braid groups; they rely on garside-like structures in \cite{bessis-garside-arxiv}.

\noindent However, we should warn the reader that this does not imply the existence of nice sections from $W$ to $B(W)$. Moreover, even for Coxeter groups, where the existence of simple systems allows us to lift $W$ in $B^+(W)$, these lifts do not satisfy the previous properties. That is, conjugate regular elements (in particular, Coxeter elements) lift to not necessarily conjugate elements in $B(W)$.
\end{remark}

If $\bm g$ is a $d^{\text{th}}$ root of the full twist, Thm.~\ref{Thm: The abelianization of B(W)} and Corol.~\ref{Cor: the abelianization of the full twist} imply that $l_{\mathcal{C}}(\bm g)\cdot d=e_{\mathcal{C}}\omega_{\mathcal{C}}$. This proves the following as in \cite[Thm.~3.2]{reiner-delMas-Hameister} (but see also \cite[Prop.~5.17:(2)]{broue-book-braid-groups}):

\begin{corollary}\label{Corol: regular numbers divide e_co_c}
For any orbit $\mathcal{C}\in\mathcal{A}/W$, a regular number $d$ always divides the quantity $e_{\mathcal{C}}\cdot\omega_\mathcal{C}$.
\end{corollary}

In Sections \ref{Section: Frobenius lemma via Coxeter numbers} and \ref{Section: The weighted enumeration} we prove some structural results for factorization enumeration formulas for arbitrary regular elements. When the order of these elements equals the highest fundamental degree $d_n$, this structural information is in fact sufficient to determine explicit formulas. We list here the corresponding types:

\begin{proposition}\cite[Prop.~4.1]{bes-zariski}\label{Prop: When d_n is regular}
Let $W$ be an irreducible complex reflection group and let $d_n$ be its largest degree. Then, $d_n$ is a regular number precisely when $W$ is a Coxeter group, or $G(r,1,n)$, $G(r,r,n)$ and $G(2r,2,2)$, or {\it any} exceptional group other than $G_{15}$.
\end{proposition} 

\begin{remark}\label{Remark: We can disregard the basepoint}
We have tried to carefully show in this section that the choice of the basepoint $v\in V^{\op{reg}}$ does not affect the theorems regarding the full twist, the abelianization, and the regular elements. At this point we choose a basepoint $v$, once and for all, and in what follows we consider the surjection $B(W)\twoheadrightarrow W$ in \eqref{1PBW1} fixed.
\end{remark}

\section{Frobenius lemma via Coxeter numbers}
\label{Section: Frobenius lemma via Coxeter numbers}

The lemma of Frobenius, which does in fact go back to Frobenius and 1896 \cite{frobenius-collected-works}, gives a representation theoretic formula for enumerating factorizations of group elements, when the factors belong to given (unions of) conjugacy classes:

\begin{theorem}\cite[App.~A.1.3]{LZ-graphs-on-surfaces} \label{Thm: Frobenius lemma}
Let $G$ be a finite group and $A_i\subset G$, $i=1\dots l$, subsets that are closed under conjugation. Then the number of factorizations $t_1\cdots t_l=g$ of an element $g\in W$, where each factor $t_i$ belongs to $A_i$, is given by$$
\dfrac{1}{|G|}\sum_{\chi\in \widehat{G}}\chi(1)\cdot\chi(g^{-1})\cdot\dfrac{\chi(A_1)}{\chi(1)}\cdots\dfrac{\chi(A_l)}{\chi(1)},$$ where $\widehat{G}$ denotes the (complete) set of irreducible characters of $G$ and $\chi(A):=\sum_{g\in A}\chi(g)$.  
\end{theorem}

For a reflection group $W$, the set of reflections $\mathcal{R}$ is indeed closed under conjugation. This lemma of Frobenius implies then a simple finite-sum form for the exponential generating function of reflection factorizations of elements of $W$. If we write $\op{Fact}_{W,g}(l)$ for the number of such factorizations of length $l$, i.e.:$$
\op{Fact}_{W,g}(l):=\#\{(t_1,\cdots,t_l)\in\mathcal{R}^l\ |\ t_1\cdots t_l=g\},$$ then the lemma of Frobenius implies that $$
\op{Fact}_{W,g}(l)=\dfrac{1}{|W|}\sum_{\chi\in\widehat{W}}\chi(1)\cdot\chi(g^{-1})\cdot \Big[\frac{\chi(\mathcal{R})}{\chi(1)}\Big]^l.$$ After this, the exponential generating function for reflection factorizations of $g$ is given by:\begin{equation}
\op{FAC}_{W,g}(t):=\sum_{l\geq 0}\op{Fact}_{W,g}(l)\cdot\frac{t^l}{l!}=\dfrac{1}{|W|}\sum_{\chi\in\widehat{W}}\chi(1)\cdot\chi(g^{-1})\cdot \op{exp}\Big[t\cdot\frac{\chi(\mathcal{R})}{\chi(1)}\ \Big].\label{Eq: Frob formula no Cox} \end{equation}Notice that, remarkably, this observation that such generating functions will be expressible as finite sums of exponentials appears already in Hurwitz's paper \cite[\S\,3:(15)]{hurwitz-1901-paper}.

Now, a priori the evaluations $\chi(\mathcal{R})$ are complex numbers, but the special structure of the set of reflections $\mathcal{R}$ forces them to in fact be integers (recall that $\mathcal{A}$ denotes the set of fixed  hyperplanes):

\begin{proposition}\label{Prop: bounds for chi(R)}
The numbers $\chi(\mathcal{R})$ are integers, and they further satisfy: $$-|\mathcal{A}|\cdot\chi(1)\ \leq\ \chi(\mathcal{R})\  \leq\ |\mathcal{R}|\cdot\chi(1).$$ Both bounds are achieved; the higher only for the trivial representation, and the lower at least for the $\op{det}$ representation.
\end{proposition}
\begin{proof}
Recall the decomposition of the set of reflections with respect to their fixed hyperplanes $H\in\mathcal{A}$ as described in \eqref{EQ: Reflection set partition into cyclis}. Keeping that notation, we choose a generator $t_H$ for each of the cyclic groups $W_H$ and write $e_H:=|W_H|$ for its order.

For each eigenvalue $\lambda$ of $t_H$ in the representation $U_{\chi}$ associated to $\chi$, the contribution of the set of reflections $\{t_H,\cdots,t_H^{e_H-1}\}$ in the evaluation of $\chi(\mathcal{R})$ equals $\sum_{k=1}^{e_H-1}\lambda^k$. Since $\lambda^{e_H}=1$, this quantity is either $e_H-1$ or $-1$ depending on whether $\lambda$ itself is $1$ or not. 

This implies the first two statements of the proposition, after noticing that the multiset of eigenvalues of $t_H$ acting on $U_{\chi}$ has $\chi(1)$-many elements. In particular, in order to recover the second inequality we use that $\sum_{H\in\mathcal{A}}(e_H-1)=|\mathcal{R}|$ which is immediate after the partitioning \eqref{EQ: Reflection set partition into cyclis}.

For the last statement, the higher bound is achieved when each eigenvalue of each $t_H$ equals $1$; of course this happens only in the trivial representation. For the lower bound, we need all $\lambda\neq 1$, which happens for instance in the ($1$-dimensional) det representation. 
\end{proof}

The character values $\chi(\mathcal{R})$ on the sum of reflections are related to an statistic of the associated representation called the {\it Coxeter number} and denoted by $c_{\chi}$. We postpone to \S\ref{Section: Malle's character permutations and the technical lemma} the discussion about its origin and for now we only give the definition:

\begin{definition}\cite[\S1.3]{gordon-griffeth-catalan-numbers}\label{Defn: Coxeter numbers}
We define the Coxeter number $c_{\chi}$ associated to the character $\chi$, as the normalized trace of the central element $\sum_{t\in\mathcal{R}}(\mathbf{1}-t)$. That is, $$c_{\chi}:=\dfrac{1}{\chi(1)}\cdot\big(|\mathcal{R}|\chi(1)-\chi(\mathcal{R})\big)=|\mathcal{R}|-\dfrac{\chi(\mathcal{R})}{\chi(1)}.$$ After Prop.~\ref{Prop: bounds for chi(R)} the numbers $c_{\chi}$ are rational, but as we will see in Corol.~\ref{Corol: local cox num are integers andbounds} they are in fact integers.
\end{definition}


It is easy now to reinterpret formula \eqref{Eq: Frob formula no Cox} in terms of the Coxeter numbers $c_{\chi}$. We record the following as a corollary of Thm.~\ref{Thm: Frobenius lemma}:

\begin{corollary}
The exponential generating function $\op{FAC}_{W,g}(t)$ for arbitrary length reflection factorizations of an element $g\in W$ is given by:\begin{equation}
\op{FAC}_{W,g}(t)=\dfrac{e^{t|\mathcal{R}|}}{|W|}\sum_{\chi\in\widehat{W}}\chi(1)\cdot \chi(g^{-1})\cdot e^{-t\cdot c_{\chi}}.\label{EQ: Frobenius via Coxeter numbers}\end{equation}
\end{corollary}

The following lemma is the main technical ingredient for the proof of Thm.~\ref{Thm: Main, not weighted}. Its derivation, which we postpone until Section~\ref{Section: Hecke algebras and the technical lemma} (see after Prop.~\ref{Prop: The key technical lemma}), relies on a cyclic action on the set $\op{Irr}(W)$ of irreducible representations of $W$ which is induced by a Galois action (see Defn.~\ref{Defn: Malle's character permutations}) on the modules of the Hecke algebra. Recall Defn.~\ref{Defn: regular element} for the concept of a regular element.

\begin{lemma}\label{Lemma: only multiples of |g| contribute}
For a complex reflection group $W$, and a \textbf{regular} element $g\in W$, the total contribution in \eqref{EQ: Frobenius via Coxeter numbers} of those characters $\chi\in\widehat{W}$ for which $c_{\chi}$ is not a multiple of $|g|$ is $0$.
\end{lemma}

The following is an essentially immediate application of Lemma~\ref{Lemma: only multiples of |g| contribute}. We state it as a theorem as all explicit formulas that come after (\ref{Corol: first of theorem}-\ref{Corol: last of theorem}) are derived as its corollaries:

\begin{theorem}\label{Thm: Main, not weighted}
For a complex reflection group $W$, and a regular element $g\in W$, the exponential generating function $\op{FAC}_{W,g}(t)$ of reflection factorizations of $g$ takes the following form:$$
\op{FAC}_{W,g}(t)=\dfrac{e^{t|\mathcal{R}|}}{|W|}\cdot \Big[ (1-X)^{l_R(g)}\cdot\Phi(X)\Big]\Big|_{X=e^{-t|g|}}\ .$$ Here $l_R(g)$ is the reflection length of $g$ and $\Phi(X)$ is a polynomial in $X$ that has degree $\frac{|\mathcal{R}|+|\mathcal{A}|}{|g|}-l_R(g)$, is not further divisible by $(1-X)$, and has constant term equal to $1$.
\end{theorem}
\begin{proof}
After Lemma~\ref{Lemma: only multiples of |g| contribute} we only need to consider terms of the form $\chi(1)\cdot\chi(g^{-1})\cdot e^{-t\cdot k|g|},\,k\in\ZZ$ in the evaluation of \eqref{EQ: Frobenius via Coxeter numbers}. Furthermore, rephrasing Prop.~\ref{Prop: bounds for chi(R)} in terms of the Coxeter numbers (via Defn.~\ref{Defn: Coxeter numbers}) forces $k\in \{0,\dots,\frac{|\mathcal{R}|+|\mathcal{A}|}{|g|}\}$. This means that if we set $X=e^{-t|g|}$, we can rewrite \eqref{EQ: Frobenius via Coxeter numbers} as $$\op{FAC}_{W,g}(t)=\dfrac{e^{t|\mathcal{R}|}}{|W|}\cdot \tilde{\Phi}(X),$$ where $\tilde{\Phi}(X)$ is a priori a polynomial in $\CC[X]$ of degree $(|\mathcal{R}|+|\mathcal{A}|)/|g|$. The last statement of Lemma~\ref{Lemma: only multiples of |g| contribute} implies also that the constant term of $\tilde{\Phi}(X)$ is equal to $\chi_{\op{triv}}(1)\cdot\chi_{\op{triv}}(g^{-1})=1$.


Now, since $\tilde{\Phi}(X)$ essentially encodes the generating function $\op{FAC}_{W,g}(t)$, the combinatorial properties of the latter impose restrictions on its structure. In particular, consider the root factorization of the polynomial:$$
\tilde{\Phi}(X)=a(\alpha_1-X)(\alpha_2-X)\cdots (\alpha_r-X).$$

If we revert to $X=e^{-t|g|}$, each of the linear terms above has a Taylor expansion that starts with $(\alpha_i-1)+t|g|+\cdots$. This means that it contributes to the leading term of $\op{FAC}_{W,g}(t)$ either by a factor of $(\alpha_i-1)$ or by a factor of $t|g|$, depending on whether $\alpha_i$ equals $1$ or not.

On the other hand, the combinatorial definition of $\op{FAC}_{W,g}(t)$ in \eqref{Eq: Frob formula no Cox} implies that its leading term is a multiple of $t^{l_R(g)}$. Therefore, exactly $l_{R}(g)$-many of the roots of $\tilde{\Phi}$ must be equal to $1$ and this completes the proof. The statements about the degree and the constant term follow from the analogous results for $\tilde{\Phi}$ described previously. 
\end{proof}

\begin{remark}
In the previous argument, the existence of a reflection length and therefore the knowledge that the first few terms of the generating function $\op{FAC}_{W,g}(t)$ are zero, came for free but was very useful nonetheless. This sort of reasoning has appeared already in \cite[end of proof of Thm.~2]{fpsac-paper-refl-length-trick}. It is hoped that similar ideas might apply to other groups with natural length functions, such as $\op{GL}_n(\mathbb{F}_q)$ (see \cites{reiner-lewis-stanton-singer-cycles}{lewis-morales-gln-fq-factorizations}). Moreover, one might construct special length functions to support different enumerative questions (as we pursue in Prop.~\ref{Prop: transitive factorizations enum.} and in Defn.~\ref{Defn: special reflection lengths n_C}).  
\end{remark}

\begin{corollary}\label{Corol: first of theorem}
For a complex reflection group $W$, and a regular element $g\in W$, the number of reduced reflection factorizations of $g$ is an integer multiple of the quantity $ \dfrac{|g|^{l_R(g)}(l_R(g))!}{|W|}.$
\end{corollary}
\begin{proof}
The leading coefficient of $\op{FAC}_{W,g}(t)$ is given, after Thm.~\ref{Thm: Main, not weighted}, by $$\Phi(1)\cdot\dfrac{|g|^{l_R(g)}(l_R(g))!}{|W|}.$$ It suffices then, to show that $\Phi(1)$ is an integer. By definition, the coefficients of the polynomial $\tilde{\Phi}(X)$ are algebraic integers and so the same is true for $\Phi(X)$. The quantity $\Phi(1)$ is thus an algebraic integer, and since it also has to be a rational number (because an integer multiple of it enumerates factorizations), it must be an integer.
\end{proof}

\begin{corollary}\label{Cor: Chapuy-Stump for dn}
For a complex reflection group $W$ and a regular element $g\in W$ of order $|g|=d_n$, the exponential generating function for reflection factorizations of $g$ is given by:$$
\op{FAC}_{W,g}(t)=\dfrac{e^{t|\mathcal{R}|}}{|W|}\cdot\big(1-e^{-t|g|}\big)^{l_R(g)}.$$ 
\end{corollary}
\begin{proof}
After Thm.~\ref{Thm: Main, not weighted} it is sufficient to show that for such an element $g$, the polynomial $\Phi(X)$ is equal to the scalar $1$, or equivalently that its degree is $0$ (notice that then, $\Phi(X)$ cannot be any other scalar since, again by Thm.~\ref{Thm: Main, not weighted}, its constant term is always $1$). 

The degree of $\Phi(X)$ is also given in the theorem; it equals $\frac{|\mathcal{R}|+|\mathcal{A}|}{|g|}-l_R(g)$. Now, Bessis has shown \cite[Prop.~4.2]{bes-zariski} that when $d_n$ is a regular number, the quantity $(|\mathcal{R}|+|\mathcal{A}|)/d_n$ is equal to the minimum number of reflections needed to generate $W$ (either $n$ or $n+1$). Therefore, if the degree of $\Phi(X)$ is not $0$, the $d_n$-regular element $g$ must  live in a reflection subgroup $W'$ of $W$.

If this were indeed the case, $g$ would still be regular in $W'$ and Springer's theorem \cite[\S32-2]{kane-book-reflection-groups} would allow us to list its eigenvalues in two ways:$$
\{\zeta^{1-d_1},\cdots,\zeta^{1-d_n}\}=\{\zeta^{1-d_1'},\cdots,\zeta^{1-d_n'}\},$$ where the $d_i'$ are the invariant degrees of $W'$ and $\zeta$ is a primitive $d_n$-th root of unity. This would force the two (multi-)sets of residues $\{d_i \op{mod}(d_n)\}$ and $\{d_i'\op{mod}(d_n)\}$ to be equal, but since $0\leq d_i\leq d_n$ and $\prod_{i=1}^n d_i=|W|>|W'|=\prod_{i=1}^n d_i'$, this is impossible.
\end{proof}

\begin{remark}
When $W$ is a well-generated group and $c$ a Coxeter element of $W$, we always have $|c|=d_n$. The previous corollary therefore completes a proof of the Chapuy-Stump formula \eqref{EQ: Chapuy-Stump formula} and extends it to the groups listed in Prop.~\ref{Prop: When d_n is regular}.
\end{remark}

In Thm.~\ref{Thm: Main, not weighted} the knowledge of the reflection length of an element provides structural information for a factorization enumeration formula. Here, we show an example where we can push this slightly further by considering a different length function, namely the transitive factorization length:

\begin{proposition}\label{Prop: transitive factorizations enum.}
The exponential generating function for transitive reflection factorizations of the regular element $g=(12\cdots n-1)(n)\in S_n$ is given by $$\op{TR-FAC}_{S_n,g}(t)=\dfrac{e^{t\binom{n}{2}}}{n!}\cdot \big( 1-e^{-t(n-1)}\big)^{n}. $$
\end{proposition}
\begin{proof}
Since $S_{n-1}$ is the only reflection subgroup of $S_n$ that contains the element $g$, we can enumerate the transitive reflection factorizations of the latter by subtracting from all possible factorizations, those that live in $S_{n-1}$:$$
\op{TR-FAC}_{S_n,g}(t)=\op{FAC}_{S_n,g}(t)-\op{FAC}_{S_{n-1},g}(t).$$ If we apply Thm.~\ref{Thm: Main, not weighted} and Corol.~\ref{Cor: Chapuy-Stump for dn} to the two terms above, we get, for $X=e^{-t(n-1)}$ :
\begin{align*}
\op{FAC}_{S_n,g}(t)-\op{FAC}_{S_{n-1},g}(t)\ &= \ \dfrac{e^{t\binom{n}{2}}}{n!}\cdot \big(1-e^{-t(n-1)}\big)^{n-2}\cdot\Phi(X) \ - \ \dfrac{e^{t\binom{n-1}{2}}}{(n-1)!}\cdot \big( 1-e^{-t(n-1)}\big)^{n-2}  \\
&= \dfrac{e^{t\binom{n}{2}}}{n!}\cdot\big(1-e^{-t(n-1)}\big)^{n-2}\cdot\big( \Phi(X)-nX\big),
\end{align*}where $\Phi(X)$ has degree $2=\frac{2\binom{n}{2}}{n-1}-(n-2)$ and constant term equal to $1$.

Notice now that the leading term of the generating function $\op{TR-FAC}_{S_n,g}(t)$ needs to be a multiple of $t^n$. Indeed, $n$ is a lower bound for the length of transitive reflection factorizations of $g$, since at least $n-1$ reflections are needed to generate $S_n$, but since also $g$ cannot be written as a product of $n-1$ reflections as it has parity $(-1)^{n-2}$.

Of course, $\big( 1-e^{-t(n-1)}\big)^{n-2}$ contributes a factor of $t^{n-2}$ to the leading term of the generating function, so $\big(\Phi(X)-nX\big)$ must contribute a multiple of $t^2$. As in the proof of Thm.~\ref{Thm: Main, not weighted}, and because $\op{deg}(\Phi(X))=2$, this implies that $$\Phi(X)-nX=(1-X)^2,$$ which completes the argument.
\end{proof}

\begin{corollary}\label{Corol: last of theorem}
For the symmetric group $S_n$ and the regular element $g=(12\cdots n-1)(n)\in S_n$, the polynomial $\Phi(X)$ from Thm.~\ref{Thm: Main, not weighted} is given by:$$
\Phi(X)=1+(n-2)X+X^2.$$
\end{corollary}

\begin{remark}
It is not clear whether one should expect a nice formula for the polynomials $\Phi_g(X)$. They don't seem to factor in small order terms and their coefficients, although integers, are not always positive (an example being the regular class of order $3$ in $E_6$). It might be however that a better answer exists for the infinite family $G(r,p,n)$ (or even just the symmetric group $S_n$), where the regular elements have simple cycle types.
\end{remark}

\begin{question}
For Weyl groups $W$, one can easily see \cite[Prop.~4.10]{springer-regular-elements} that any regular element of order $d$ divides the set of roots in orbits of size $d$.   Perhaps this could be used in a fashion similar to the recursion in \cites{deligne-letter-to-looijenga}{reading-chains-in-noncrossing} and, possibly assuming the Lemma of Frobenius \eqref{EQ: Frobenius via Coxeter numbers}, give a combinatorial proof of our technical Lemma~\ref{Lemma: only multiples of |g| contribute}.
\end{question}

\section{Hecke algebras and the technical lemma}
\label{Section: Hecke algebras and the technical lemma}

Iwahori-Hecke algebras associated to Weyl groups $W$ appear naturally as endomorphism algebras of certain induced modules in the representation theory of finite groups of Lie type. They can also be seen as deformations of the corresponding group ring $\ZZ[W]$. This second interpretation has been extended for all complex reflection groups:

Let $\mathcal{C}\in\mathcal{A}/W$ denote an orbit of hyperplanes, and $e_{\mathcal{C}}$ the common order of the pointwise stabilizers $W_H$ (for $H\in\mathcal{C}$). Consider now a set of $\sum_{\mathcal{C}\in\mathcal{A}/W}e_{\mathcal{C}}$ many variables $\bm u:=(u_{\mathcal{C},j})_{(\mathcal{C}\in\mathcal{A}/W),(0\leq j\leq e_{\mathcal{C}}-1)}$ and write $\ZZ[\bm u,\bm u^{-1}]$ for the Laurent polynomial ring on the $u_{\mathcal{C},j}$'s. 

\begin{definition}\cite[Defn.~4.21]{broue-malle-rouquier-braid-hecke}\label{Defn: Hecke algebra}
The {\it generic Hecke algebra} $\mathcal{H}(W)$ associated to $W$ is the quotient of the group ring $\ZZ[\bm u,\bm u^{-1}]B(W)$ of the braid group, over the ideal generated by the elements of the form \begin{equation}
(\bm s-u_{\mathcal{C},0})(\bm s-u_{\mathcal{C},1})\cdots (\bm s-u_{\mathcal{C},e_{\mathcal{C}}-1}),\label{EQ: deformed order relations}\end{equation}
which we call {\it deformed order relations} (see \eqref{EQ: specialization of the u's}). Here $\bm s$ runs over all possible braid reflections (see \S\ref{Section: Braid groups and generators of the monodromy}) around the stratum $\mathcal{C}$ of $\mathcal{H}$. Notice that for each orbit $\mathcal{C}$ one such relation is in fact sufficient since all corresponding elements $\bm s_{\mathcal{C},\gamma}$ are conjugate in $B(W)$. 
\end{definition}

\begin{notation}
For an element $\bm g$ of the braid group $B(W)$, we denote the corresponding element in the Hecke algebra by $T_{\bm g}$.
\end{notation}

Any ring map $\theta:\ZZ[\bm u,\bm u^{-1}]\rightarrow R$ defines an $R$-module structure on the Hecke algebra. We write $\mathcal{H}_R(W):=\mathcal{H}(W)\otimes_{\ZZ[\bm u,\bm u^{-1}]}R$ and call $\mathcal{H}_R(W)$  a {\it \color{blue} specialization} of $\mathcal{H}(W)$. The map $\theta$ induces thus a canonical map $\tilde{\theta}: \mathcal{H}(W)\rightarrow \mathcal{H}_R(W)$ via $T_{\bm g}\rightarrow T_{\bm g}\otimes 1$. 

The Hecke algebra is by construction a deformation of the group algebra of $W$.  Indeed, the specialization (recall $\zeta_n:=\op{exp}(2\pi i /n)$) \begin{equation}
u_{\mathcal{C},j}\xrightarrow{\displaystyle \ \sigma \ } \zeta_{e_{\mathcal{C}}}^j \label{EQ: specialization of the u's}\end{equation} transforms the defining relations \eqref{EQ: deformed order relations} to order relations of the form $\bm s^{e_{\mathcal{C}}}=1$. Then, by Prop.~\ref{Prop: (braid reflections)^e_c generate P(W)} $\mathcal{H}(W)$ reduces to the group ring $\ZZ[(\zeta_{e_{\mathcal{C}}})]_{(\mathcal{C}\in\mathcal{A}/W)}[W]$ and the map $\tilde{\sigma}$ agrees with the fixed (see Rem.~\ref{Remark: We can disregard the basepoint}) surjection $B(W)\twoheadrightarrow W$. That is, if $g\in W$ is the image of $\bm g\in B(W)$ under \eqref{1PBW1}, then $$\tilde{\sigma}(T_{\bm g})=g.$$

\begin{definition}\label{Defn: admissible specializations} A specialization $\theta$ will be called {\it \color{blue} admissible} if it factors through \eqref{EQ: specialization of the u's}; in other words if there is a map $f:R\rightarrow \ZZ[(\zeta_{e_{\mathcal{C}}})]$ such that $f\circ\theta(u_{\mathcal{C},j})=\zeta_{e_{\mathcal{C}}}^j$. 
\end{definition}

Two particular specializations are fundamental in what follows. We first pick a set of parameters $\bm x:=(x_{\mathcal{C}})_{\mathcal{C}\in\mathcal{A}/W}$ and the single parameter $x$ and define the following ring maps:\begin{align}
\theta_{\bm x}:\ZZ[\bm u,\bm u^{-1}]&\rightarrow \ZZ[\bm x,\bm x^{-1}] \quad\quad \quad\quad\text{ and }& \theta_x:\ZZ[\bm u,\bm u^{-1}]&\rightarrow \ZZ[x,x^{-1}]\nonumber \\
\theta_{\bm x}(u_{\mathcal{C},j})&=\begin{cases} x^{\phantom{N_W}}_{\mathcal{C}}&\text{if }j=0\\ \zeta_{e_{\mathcal{C}}}^j&\text{if }j\neq 0\end{cases}  &\theta_x(u_{\mathcal{C},j})&=\begin{cases} x&\text{if }j=0\\ \zeta_{e_{\mathcal{C}}}^j&\text{if }j\neq 0\end{cases} \label{EQ: admissible specializations}
\end{align}
Both $\theta_{\bm x}$ and $\theta_x$ are admissible specializations (as seen by further sending $x_{\mathcal{C}}$ or $x$ to $1$). We write $\mathcal{H}_{\bm x}(W)$ and $\mathcal{H}_x(W)$ for the corresponding Hecke algebras, while noting that the latter is the analogue of the $1$-parameter Iwahori-Hecke algebra of real reflection groups $W$.

\subsubsection*{Artin-like presentations and the BMR-freeness theorem}

Bessis \cite{bes-zariski} has shown that the braid groups $B(W)$ always have ``Artin-like" presentations. These are presentations of the form $$\langle \bm s_1,\cdots, \bm s_n\ |\ p_j(\bm s_1,\cdots \bm s_n)= q_j(\bm s_1,\cdots \bm s_n)\rangle,$$ where the $\bm s_i$'s are braid reflections (so they equal $\bm s_{\mathcal{C},\gamma}$ for suitable $\mathcal{C}$ and $\gamma$) and their images $s_H \in W$ form a minimal generating set of (distinguished) reflections. Furthermore, the relations $(p_j,q_j)$ encode positive words of equal length in the $\bm s_i$'s and are such so that by adding the order relations $\bm s_i^{e_H}=1$, one obtains a presentation of the group $W$.

By now, such Artin-like presentations have been found for all braid groups $B(W)$ (see \cite[Appendix~A.2]{broue-book-braid-groups}). With access to these, one can write down explicit presentations for the Hecke algebras and with them attempt to study their various structural properties and invariants.

\begin{example}\label{Example Hecke algebra of G_26}
The generic Hecke algebra of $G_{26}$ (over the ring $\ZZ[x_0^{\pm 1},\cdots y_2^{\pm 1}]$) is:
\begin{align*}
\mathcal{H}(G_{26})\ = \ \langle\ \bm s,\bm t,\bm u\ |\ &\bm{stst}=\bm{tsts},\ \bm{su}=\bm{us},\ \bm{tut}=\bm{utu},\\
&(\bm s-x_0)(\bm s-x_1)=0\\
&(\bm t\,-y_0)(\bm t\,-\,y_1)(\bm t\,-y_2)=0\\
&(\bm u-y_0)(\bm u-y_1)(\bm u-y_2)=0\ \rangle
\end{align*}
The braid reflections $\bm t$ and $\bm u$ are conjugate (although this is a bit hard to see from the given presentation of $B(G_{26})$), so we use the same set of variables for their deformed order relations. After the specializations $(x_0,x_1)=(1,-1)$, $(y_0,y_1,y_2)=(1,\zeta_3,\zeta_3^2)$,  we obtain the following Coxeter-like presentation of $G_{26}$:$$ G_{26}\ =\ \langle s,t,u\ |\ stst=tsts,\ su=us,\ tut=utu,\ s^2=t^3=u^3=1\ \rangle.$$
\end{example}

This definition of Hecke algebras, which recovers the usual Iwahori-Hecke algebras when $W$ is a Coxeter group, is due to Brou\'e, Malle, and Rouquier, and was introduced in their seminal paper \cite{broue-malle-rouquier-braid-hecke}. There, they also made various conjectures about these Hecke algebras, the most important of which was until recently known as \textbf{\textit{``The BMR freeness conjecture"}}:

\begin{theorem*}\cite{etingof-proof-of-BMR-char-0}\cite[after Thm.~3.5]{chlouveraki-et-al-arxiv-BMR}

The algebra $\mathcal{H}(W)$ is a free $\ZZ[\bm u,\bm u^{-1}]$-module of rank $|W|$.
\end{theorem*} 

\subsection{Tits' deformation theorem for admissible specializations}

For this work, the first important consequence of the BMR-freeness theorem is that it determines, via Tits' deformation theorem, a  bijection between the  irreducible representations of $W$ and those of the Hecke algebra. The reader might refer to  \cite[\S7]{geck-pfeiffer-book} for proofs and terminology. 

To apply Tits' deformation theorem, we first have to move to split extensions of $\mathcal{H}(W)$ and of the group algebra of $W$. For the latter, we could simply work over $\CC[W]$, but it takes little effort to describe its minimal splitting field. To begin with, it is easy to see \cite[Corol.~3.2]{bessis-corps-de-definition} that the reflection representation $V$ of $W$ can be realized over the field $K$ generated by the traces of the elements of $W$ on $V$. It is a theorem of Benard and Bessis \cites{benard-unitary}{bessis-corps-de-definition} that in fact all representations of $W$ can be realized over $K$.

We henceforth call $K$ the {\it field of definition} of $W$; it equals $\QQ$ when $W$ is a Weyl group and satisfies $K\leq \mathbb{R}$ when $W$ is a finite Coxeter group. One might then hope that $K(\bm u)$ is a splitting field for $\mathcal{H}(W)$. Although this is not the case, the answer is only slightly more complicated. Assuming the BMR-freeness conjecture, Malle proved (with further case-specific arguments, but see \S\ref{Section: On the uniformity of the proofs}):

\begin{proposition}\cite[Thm.~5.2]{malle-fake-degrees} \label{Prop. Malle's charact. of splitting field L of H}
Let $K$ be the field of definition of $W$ as above. Then, there exists a number $N_W$ such that for a set of parameters $\bm v:=(v_{\mathcal{C},j})_{(\mathcal{C}\in\mathcal{A}/W),(0\leq j\leq e_{\mathcal{C}}-1) }$ that satisfy $$
v_{\mathcal{C},j}^{N_W}=\zeta_{e_{\mathcal{C}}}^{-j}u^{\phantom{N_W}}_{\mathcal{C},j},$$ the field $K(\bm v)$ is a splitting field for $\mathcal{H}(W)$. We write $\mathcal{H}_{K(\bm v)}(W):=\mathcal{H}(W)\otimes_{\ZZ[\bm u,\bm u^{-1}]}K(\bm v)$.
\end{proposition}

Of course, after the BMR-freeness conjecture, $\mathcal{H}_{K(\bm v)}(W)$ will also be a free $K[\bm v, \bm v^{-1}]$-module and we may extend the specialization \eqref{EQ: specialization of the u's} to a map $K[\bm v, \bm v^{-1}]\rightarrow K$, which we also call $\sigma$ and is given by\begin{equation}
v^{\phantom{N_W}}_{\mathcal{C},j}\xrightarrow{\displaystyle \ \sigma\ } 1. \label{EQ: Specialization v-> 1}
\end{equation} Notice that, just as in \eqref{EQ: specialization of the u's}, the induced map $\tilde{\sigma}:\mathcal{H}_{K(\bm v)}(W)\rightarrow K[W]$ agrees with the fixed surjection $B(W)\twoheadrightarrow W$. The freeness over $\ZZ[\bm v, \bm v^{-1}]$, the fact that $K(\bm v)$ and $K$ are splitting fields for $\mathcal{H}(W)$ and $W$ respectively, and the semisimplicity of $K[W]$, constitute the assumptions of Tits' deformation theorem (see \cite[\S7.3-4]{geck-pfeiffer-book}). Its conclusion is then:

\begin{theorem}\label{Thm: Tits deformation theorem}
The algebra $\mathcal{H}_{K(\bm v)}(W)$ is also semisimple and the specialization map $\sigma$ induces a bijection $$ \op{d}_{\sigma}:\op{Irr}\big(\mathcal{H}_{K(\bm v)}(W)\big) \xrightarrow{\displaystyle\sim} \op{Irr}\big(K[W]\big), $$ between the irreducible modules of the two algebras, that respects the spectra of elements. That is, if $U$ and $\op{d}_{\sigma}(U)$ are irreducible modules matched by $\op{d}_{\sigma}$, then the following diagram commutes: 
\begin{equation}%
\begin{tikzcd}
\mathcal{H}_{K(\bm v)}(W)\ni T_{\bm g} \arrow{rrr}{\displaystyle\quad  \mathfrak{p}_U\quad} \arrow[swap,xshift=1.1cm]{d}{\displaystyle \tilde{\sigma}} &&& K[\bm v,\bm v^{-1}][X] \arrow[xshift=-0.5cm]{d}{\displaystyle t_{\sigma}} \\%
\quad \quad K[W]\ni g \arrow{rrr}{\displaystyle \quad \mathfrak{p}_{\op{d}_{\sigma}(U)}\quad  }&&& K[X]\quad \quad \quad 
\end{tikzcd}\label{EQ: spectra are respected by tits def.} 
\end{equation}
The horizontal maps $\mathfrak{p}_M$ send an element $T_{\bm g}$ or $g$ to its characteristic polynomial under the representation $M$, while the vertical maps are naturally induced by $\sigma$. In particular, since character values are determined by the spectra of elements, if $\chi_{\bm v}$ and $\chi$ are the characters associated to $U$ and $\op{d}_{\sigma}(U)$ respectively, we will have \begin{equation}
\chi(g)= \sigma\big(\chi_{\bm v}(T_{\bm g})\big).\label{EQ: characters are respected by tits def.}
\end{equation}
\end{theorem}
\begin{proof}[Remark:]
It is not a priori clear that the characteristic polynomials of elements $T_{\bm g}$ live in $K[\bm v,\bm v^{-1}][X]$ (instead of just $K(\bm v)[X]$); this is shown in \cite[Prop.~7.3.8]{geck-pfeiffer-book}. The existence of the map $d_{\sigma}$ and that it respects spectra is proved in [ibid, Thm.~7.4.3], and the fact that it is a bijection in [ibid, Thm.~7.4.6].
\end{proof}

We can apply Tits' deformation theorem on any admissible (see Defn.~\ref{Defn: admissible specializations}) specialization of $\mathcal{H}(W)$ by first moving to a splitting field as prescribed by Prop.~\ref{Prop. Malle's charact. of splitting field L of H}. In particular, for the algebras $\mathcal{H}_{\bm x}(W)$ and $\mathcal{H}_x(W)$ from \eqref{EQ: admissible specializations}, the corresponding splitting fields have to be $K(\bm y)$ and $K(y)$ respectively for parameters $\bm y:=(y_{\mathcal{C}})_{\mathcal{C}\in\mathcal{A}/W}$ and $y$ that satisfy $y_{\mathcal{C}}^{N_W}=x_\mathcal{C}$ and $y^{N_W}=x$.

Now Thm.~\ref{Thm: Tits deformation theorem} implies that we can simultaneously index the characters of $\mathcal{H}(W)$, $\mathcal{H}_{\bm x}(W)$, and $\mathcal{H}_x(W)$ by characters $\chi\in\widehat{W}$. Indeed, if say $f_{\bm x}$ is the factoring morphism of Defn.~\ref{Defn: admissible specializations}, we have \begin{equation}\op{Irr}\big(\mathcal{H}(W)\big)\xrightarrow{\displaystyle \ d_{\theta_{\bm x}}\ }\op{Irr}\big(\mathcal{H}_{\bm x}(W)\big)\xrightarrow{\displaystyle \ d_{f_{\bm x}}\ }\op{Irr}\big(K[W]\big),\label{EQ: admissible specializations and Tits factoring}\end{equation} where $d_{\theta_{\bm x}}$ and $d_{f_{\bm x}}$ are bijections which satisfy $d_{\sigma}=d_{\theta_{\bm x}}\circ d_{f_{\bm x}}$ and moreover respect spectra as in \eqref{EQ: spectra are respected by tits def.}. We will therefore denote the characters of the three Hecke algebras by $\chi_{\bm v}, \chi_{\bm y}$, and $\chi_y$ respectively, using the parameters $\bm v,\bm y,y$ that define the splitting fields.

\begin{definition}\label{Defn: rational character}
We say that a character of the Hecke algebra $\mathcal{H}(W)$ is {\it rational} with respect to the specializations $\theta_{\bm x}$ or $\theta_x$ (respectively {\it generically rational}) if its values lie in $K(\bm x)$ or $K(x)$ (respectively in $K(\bm u)$), as opposed to the splitting fields. Similarly we talk of a rational spectrum of some element $T_{\bm g}$ for a given representation and specialization.
\end{definition}

\begin{remark}
Notice that a character might be rational for the specialization $\theta_x$ but not for $\theta_{\bm x}$. This is for instance the case when a monomial of the form $\sqrt{x^{\phantom{N_W}}_{\mathcal{C},0}x^{\phantom{N_W}}_{\mathcal{C'},0}}$ appears as its value (which is not rational for $\theta_{\bm x}$ but becomes $x$ for $\theta_x$). For example, the group $G_6$ has $6$ characters that are not generically rational (see \cite[Table~8.1]{malle-fake-degrees}) but a CHEVIE \cites{chevie-all}{chevie-michel} calculation shows only $2$ irrational characters for $\theta_x$.
\end{remark}

\subsection{Character values on roots of the full twist}
\label{Section: Character values on roots of the full twist}

For a character $\chi_{\bm v}$ of the generic Hecke albegra $\mathcal{H}_{K(\bm v)}(W)$, let $m^{\chi_{\bm v}}_{\mathcal{C},j}$ denote the multiplicity of $u_{\mathcal{C},j}$ as an eigenvalue of any braid reflection $\bm s_{\mathcal{C},\gamma}$ in the representation $U$ associated with $\chi_{\bm v}$. After Tits' deformation theorem (in particular, after \eqref{EQ: spectra are respected by tits def.}) this equals the multiplicity of $\zeta_{e_{\mathcal{C}}}^j=\sigma (u_{\mathcal{C},j})$ as an eigenvalue of any {\it distinguished} reflection $s_H,\ H\in\mathcal{C}$, in the representation $\op{d}_{\sigma}(U)$. 

The same is true for any admissible speciliazation $\theta$ (notice that since $f\circ\theta (u_{\mathcal{C},j})=\zeta_{e_{\mathcal{C}}}^j$, the elements $\theta(u_{\mathcal{C},j})$ cannot be equal), so for the analogously defined numbers $m^{\chi_{\bm y}}_{\mathcal{C},j},\ m^{\chi_y}_{\mathcal{C},j},\ m^{\chi}_{\mathcal{C},j}$, we have $$  m^{\chi_{\bm v}}_{\mathcal{C},j}\ =\ m^{\chi_{\bm y}}_{\mathcal{C},j}\ =\ m^{\chi_y}_{\mathcal{C},j}\ =\ m^{\chi}_{\mathcal{C},j}.$$ In view of this, we will only use the latter notation $m^{\chi}_{\mathcal{C},j}$ from now on. Notice finally that by the defining relations \eqref{EQ: deformed order relations}, the only possible eigenvalues for any $\bm s_{\mathcal{C},\gamma}$ are precisely the $u_{\mathcal{C},j}$'s. We therefore have (for any $\mathcal{C}\in\mathcal{A}/W$) \begin{equation} \sum_{j=0}^{e_{\mathcal{C}}-1}m^{\chi}_{\mathcal{C},j}=\chi(1).\label{Eq: sum of m_H,j's=chi(1)}\end{equation}

The following proposition is essential for the proof of our technical lemma (Prop.~\ref{Prop: The key technical lemma}). For completion, we reproduce here its proof following \cite[\S\,4:~E]{broue-michel-french-paper} very closely. To simplify its statement we first introduce the following notation (recall also that $\omega_{\mathcal{C}}=|\mathcal{C}|$ for an orbit $\mathcal{C}\in\mathcal{A}/W$):

\begin{definition}\label{Def. Pre character value stuff}
Consider\footnote{We move to a larger ring, so that the roots $u_{\mathcal{C},j}^{1/\chi(1)}$ are well defined. Shortly however, Prop.~\ref{Prop: character values on roots of full twist} will show that these monomials actually live in $K[\bm v]$.} the element of $K[\bm u^{1/|W|}]$ given as $$z_{\chi_{\bm v}}(\bm \pi):=\prod_{\mathcal{C}\in\mathcal{A}/W}\prod_{j=0}^{e_{\mathcal{C}}-1}u_{\mathcal{C},j}^{(1/\chi(1))m^{\chi}_{\mathcal{C},j}e_{\mathcal{C}}\omega_{\mathcal{C}}},$$ and, for a regular number $d$ (see Defn.~\ref{Defn: regular element}), write $$z_{\chi_{\bm v}}(\bm \pi)^{1/d}:=\prod_{\mathcal{C}\in\mathcal{A}/W}\prod_{j=0}^{e_{\mathcal{C}}-1}u_{\mathcal{C},j}^{(1/d\chi(1))m^{\chi}_{\mathcal{C},j}e_{\mathcal{C}}\omega_{\mathcal{C}}}.$$
Finally, denote by $N(\chi)$ the quantity $$N(\chi):=\sum_{\mathcal{C}\in\mathcal{A}/W}\omega_{\mathcal{C}}\cdot\sum_{j=0}^{e_{\mathcal{C}}-1}jm^{\chi}_{\mathcal{C},j}.$$
\end{definition}

\begin{remark}\label{Remark: N(x) via coexponents}
$N(\chi)$ usually denotes the sum of the $\chi^{*}$-exponents (see \cite[Chapter~4:~\S4]{lehrer-taylor-book}) of the representation that affords $\chi$. This in fact agrees with the definition above (see \cite[Prop.~4.1]{broue-michel-french-paper}, or \cite[Lemma~10.15 and Remark~10.12]{lehrer-taylor-book} which includes Gutkin's theorem). We are only going to use it as a symbol (but see also Remark~\ref{Remark: cox numbers and N(x)+N(x*)}).
\end{remark}

\begin{proposition}\label{Prop: character values on roots of full twist}\cite[Prop.~4.16]{broue-michel-french-paper}\newline
For a character $\chi_{\bm v}$ of the generic Hecke algebra, the values on the full twist $T_{\bm\pi}$ are given by $$ \chi_{\bm v}(T_{\bm\pi})=\chi(1)e^{-2i\pi N(\chi)/\chi(1)}z_{\chi_{\bm v}}(\bm \pi).$$
Moreover, if $\bm w$ is a $d$-th root of some power $\bm\pi^l$ and its image in $W$ under the fixed surjection \eqref{1PBW1} is $w$, we have $$\chi_{\bm v}(T_{\bm w})=\chi(w)e^{-2i\pi lN(\chi)/d\chi(1)}z_{\chi_{\bm v}}(\bm\pi)^{l/d}. $$
\end{proposition}
\begin{proof}
We are only going to prove the second statement, which reduces to the first for $d=l=1$.

Consider the determinant character $\op{det}_{\chi_{\bm v}}$ associated to $\chi_{\bm v}$. Because it is linear, it factors through the abelianization $B^{\op{ab}}$ and its values on powers of the full twist are given after Corol.~\ref{Cor: the abelianization of the full twist} by $$\op{det}_{\chi_{\bm v}}(T_{\bm \pi}^l)=\prod_{\mathcal{C}\in\mathcal{A}/W}\prod_{j=0}^{e_{\mathcal{C}}-1}u_{\mathcal{C},j}^{m^{\chi}_{\mathcal{C},j}\omega_{\mathcal{C}}e_{\mathcal{C}}l}=z_{\chi_{\bm v}}(\bm \pi)^{\chi(1)l}.$$

Now, since $\bm\pi^l$ is central in $B(W)$ (and therefore also in $\mathcal{H}_{K(\bm v)}(W)$), it acts on irreducible representations as a scalar. That is, its spectrum is given by $$\op{Spec}_{\chi_{\bm v}}(T_{\bm\pi}^l)=\{\xi z_{\chi_{\bm v}}(\bm\pi)^l\ (\chi(1)\text{-many times})\},$$ where $\xi$ is a $\chi(1)$-th root of unity. Since $\bm w^d=\bm\pi^l$, we further have that the spectrum of $T_{\bm w}$ is $$\op{Spec}_{\chi_{\bm v}}(T_{\bm w})=\{\xi_kz_{\chi_{\bm v}}(\bm\pi)^{l/d}\ |\ (1\leq k\leq \chi(1))\},$$ where the $\xi_k$ are $d$-th roots of $\xi$, which of course means that $$\chi_{\bm v}(T_{\bm w})=z_{\chi_{\bm v}}(\bm\pi)^{l/d}\sum_{k=1}^{\chi(1)}\xi_k.$$

We are only left with computing the sum of the $\xi_k$. Notice that after Tits' deformation theorem (in particular, the statement \eqref{EQ: characters are respected by tits def.} that character values are respected), we will have for the specialization $\sigma$ of \eqref{EQ: Specialization v-> 1} that $\sigma(\chi_{\bm v}(T_{\bm w}))=\chi(w)$. Since the right hand side of the previous equation will then also have to evaluate to $\chi(w)$ under $\sigma$, we will have that $$\sum_{k=1}^{\chi(1)}\xi_k=\chi(w)\sigma\big(z_{\chi_{\bm v}}(\bm \pi)^{l/d}\big)^{-1}.$$
Finally, recalling Defn.~\ref{Def. Pre character value stuff} and that $\sigma(u_{\mathcal{C},j})=\zeta_{e_{\mathcal{C}}}^j=\op{exp}(2i\pi j/e_{\mathcal{C}})$, it is easy to see that
$$\sigma(z_{\chi_{\bm v}}(\bm \pi)^{l/d})=\op{exp}\Big(2i\pi l/(d\chi(1))\sum_{\mathcal{C}\in\mathcal{A}/W}\omega_{\mathcal{C}}\sum_{j=1}^{e_{\mathcal{C}-1}}jm^{\chi}_{\mathcal{C},j}  \Big)=e^{2i\pi lN(\chi)/d\chi(1)}.$$%
\end{proof}

By applying the specialization $\theta_{\bm x}$ from \eqref{EQ: admissible specializations} on the previous proposition, we easily get: 

\begin{corollary}\label{Corol: character calculation for H_x(W)}
Let $\bm w$ be a $d$-th root of some power $\bm \pi^l$ as above and let $\chi_{\bm y}$ be a character of the specialization $\mathcal{H}_{\bm x}(W)$ as in \eqref{EQ: admissible specializations and Tits factoring}. We have \begin{enumerate}
\item\ \vspace{-0.35cm} $$\chi_{\bm y}(T_{\bm \pi})=\chi(1)\prod_{\mathcal{C}\in\mathcal{A}/W}x_{\mathcal{C}}^{(1/\chi(1))m^{\chi}_{\mathcal{C},0}e_{\mathcal{C}}\omega_{\mathcal{C}}}.$$
\item\ \vspace{-0.35cm} $$\chi_{\bm y}(T_{\bm w})=\chi(w)\prod_{\mathcal{C}\in\mathcal{A}/W}x_{\mathcal{C}}^{(l/d\chi(1))m^{\chi}_{\mathcal{C},0}e_{\mathcal{C}}\omega_{\mathcal{C}}}.$$
\end{enumerate}
\end{corollary}

\subsection{Local Coxeter numbers}

We are now going to define a local version of Coxeter numbers (see Defn.~\ref{Defn: Coxeter numbers}) and study how they are precisely related to the exponents that appear in the character calculation of the previous Corol.~\ref{Corol: character calculation for H_x(W)}.

\begin{definition}\label{Defn: Local Coxeter numbers} 
We define the {\it local} Coxeter number $c_{\chi_{,\mathcal{C}}}$ associated to the character $\chi$ and the hyperplane orbit $\mathcal{C}\in\mathcal{A}/W$, as the normalized trace $$c_{\chi_{,\mathcal{C}}}:=\dfrac{1}{\chi(1)}\cdot\chi\Big( \sum_{V^{t}\in\mathcal{C}}(\bm 1-t)\Big).$$ Here, the sum is taken over all reflections $t$ whose fixed hyperplane $H=V^{t}$ belongs to the orbit $\mathcal{C}$. Notice that these numbers are a refinement of the Coxeter numbers in the sense that $c_{\chi}=\sum c_{\chi_{,\mathcal{C}}}$
\end{definition}

\begin{proposition}\label{Prop: local cox num formula}
The local Coxeter numbers satisfy $$c_{\chi_{,\mathcal{C}}}=e_{\mathcal{C}}\cdot\omega_{\mathcal{C}}\cdot\Big(1-\dfrac{m^{\chi}_{\mathcal{C},0}}{\chi(1)}\Big).$$
\end{proposition}
\begin{proof}
As we saw in \eqref{EQ: Reflection set partition into cyclis}, because the parabolic groups for hyperplanes are cyclic, the set of reflections can be partitioned into sets of the form $\{t_H,\cdots,t^{e_H-1}_H\}$. Moreover, recalling the definition of $m^{\chi}_{\mathcal{C},j}$ from the beginning of this section, we see that the spectrum of $t_H^k$ (for $H\in\mathcal{C}$) is given by $$\op{Spec}_{\chi}(t_H^k)=\{ \zeta_{e_{\mathcal{C}}}^{jk}\ (m^{\chi}_{\mathcal{C},j}\text{-many times})\ |\ 0\leq j\leq e_{\mathcal{C}}-1\}.$$ We can then pick an $H\in\mathcal{C}$ and a generator $t_H$ of $W_H$, and start the evaluation by computing \begin{align*}
\sum_{V^t\in\mathcal{C}}\chi(\bm 1-t)&=\chi(1)(e_{\mathcal{C}}-1)\omega_{\mathcal{C}}-\omega_{\mathcal{C}}\sum_{k=1}^{e_{\mathcal{C}}-1}\chi(t_H^k)\\
&=\chi(1)(e_{\mathcal{C}}-1)\omega_{\mathcal{C}}-\omega_{\mathcal{C}}\sum_{k=1}^{e_{\mathcal{C}}-1}\sum_{j=0}^{e_{\mathcal{C}}-1}m^{\chi}_{\mathcal{C},j}\zeta_{e_{\mathcal{C}}}^{jk}.
\end{align*}
Now, notice that the sum $\sum_{k=1}^{e_{\mathcal{C}}-1}\zeta_{e_{\mathcal{C}}}^{jk}$ equals $e_{\mathcal{C}}-1$ or $-1$ depending on whether $j=0$ or not. So, after changing the order of summation, we have \begin{align*}
\sum_{V^t\in\mathcal{C}}\chi(\bm 1-t)&=\chi(1)(e_{\mathcal{C}}-1)\omega_{\mathcal{C}}+\sum_{j=1}^{e_{\mathcal{C}}-1}\omega_{\mathcal{C}}m^{\chi}_{\mathcal{C},j}-(e_{\mathcal{C}}-1)\omega_{\mathcal{C}}m^{\chi}_{\mathcal{C},0}\\
&=\chi(1)e_{\mathcal{C}}\omega_{\mathcal{C}}-e_{\mathcal{C}}\omega_{\mathcal{C}}m^{\chi}_{\mathcal{C},0},
\end{align*}where the second equation is because of \eqref{Eq: sum of m_H,j's=chi(1)}. This completes the proof.
\end{proof}

We can now rewrite the character calculation from Corol.~\ref{Corol: character calculation for H_x(W)} replacing the quantities in the exponents with equivalent ones in terms of the Coxeter numbers $c_{\chi_{,\mathcal{C}}}$ (and via Prop.~\ref{Prop: local cox num formula}). With the notation being the same as in the statement of the Corollary, we have: \begin{equation}
\chi_{\bm y}(T_{\bm \pi})=\chi(1)\prod_{\mathcal{C}\in\mathcal{A}/W}x_{\mathcal{C}}^{e_{\mathcal{C}}\omega_{\mathcal{C}}-c_{\chi_{,\mathcal{C}}}}\quad\text{ and }\quad\chi_{\bm y}(T_{\bm w})=\chi(w)\prod_{\mathcal{C}\in\mathcal{A}/W}x_{\mathcal{C}}^{(e_{\mathcal{C}}\omega_{\mathcal{C}}-c_{\chi_{,\mathcal{C}}})l/d}.\label{EQ: character values via local Coxeter}\end{equation}

Moreover, after the further specialization $x_{\mathcal{C}}\rightarrow x$ of $\theta_x$ from \eqref{EQ: admissible specializations} and for the characters $\chi_y$ of $\mathcal{H}_x(W)$ as in \eqref{EQ: admissible specializations and Tits factoring}, we have (recalling that $\sum e_{\mathcal{C}}\omega_{\mathcal{C}}=|\mathcal{R}|+|\mathcal{A}|$ and that $\sum c_{\chi_{,\mathcal{C}}}=c_{\chi}$):
\begin{equation}
\chi_y(T_{\bm \pi})=\chi(1)\cdot x^{|\mathcal{R}|+|\mathcal{A}|-c_{\chi}}\quad\text{ and }\quad\chi_y(T_{\bm w})=\chi(w)\cdot x^{(|\mathcal{R}|+|\mathcal{A}|-c_{\chi})l/d}.\label{EQ: character values via Coxeter-not local}\end{equation}

\begin{remark}\label{Remark: cox numbers and N(x)+N(x*)}
This last equation is precisely what appears in \cite[Prop.~4.18]{broue-michel-french-paper} but with an equivalent expression for the Coxeter numbers: $$c_{\chi}=\dfrac{N(\chi)+N(\chi^*)}{\chi(1)},$$ where the numbers $N(\chi)$ are given in Defn.~\ref{Def. Pre character value stuff} (see also Rem.~\ref{Remark: N(x) via coexponents}). This expression also appears in \cite[Lemma~1]{michel-deligne-lusztig-derivation} but the statement of that Lemma might be misleading as it holds regardless of the values $e_{\mathcal{C}}$. For completion, we include the calculation: $$\chi(1)c_{\chi}=\chi(1)\sum_{\mathcal{C}\in\mathcal{A}/W}c_{\chi_{,\mathcal{C}}}=\sum_{\mathcal{C}\in\mathcal{A}/W}\omega_{\mathcal{C}}\sum_{j=1}^{e_{\mathcal{C}}-1}e_{\mathcal{C}}m^{\chi}_{\mathcal{C},j}=N(\chi)+N(\chi^*).$$
In fact, Michel later on \cite[Rem.~2]{michel-deligne-lusztig-derivation} notes that for all groups $W$ one has (see Defn.~\ref{Defn: Malle's character permutations}) $$c_{\chi}=\dfrac{N(\chi)+N(\Psi(\chi^*))}{\chi(1)},$$ which is equivalent to the first statement as $N(\Psi(\chi))=N(\chi)$ after Prop.~\ref{Prop: properties of Psi for Cox numbers}.
\end{remark}

For the proof of the integer property in the following corollary, we again follow \cite{broue-michel-french-paper} closely and reproduce the argument here for completion.

\begin{corollary}\cite[Corol.~4.17]{broue-michel-french-paper}\label{Corol: local cox num are integers andbounds}
The Coxeter numbers $c_{\chi_{,\mathcal{C}}}$ are integers and they satisfy $$ 0\leq c_{\chi_{,\mathcal{C}}}\leq e_{\mathcal{C}}\cdot\omega_{\mathcal{C}}.$$
\end{corollary}
\begin{proof}
The inequalities are immediate from Prop.~\ref{Prop: local cox num formula}, since $0\leq m^{\chi}_{\mathcal{C},0}\leq \chi(1)$. To see that the numbers $c_{\chi_{,\mathcal{C}}}$ are integers, it is enough to show that the values $\chi_{\bm y}(T_{\bm \pi})$ given in \eqref{EQ: character values via local Coxeter} belong to $K(\bm x)$ (as opposed to the splitting field $K(\bm y)$ of $\mathcal{H}_{\bm x}(W)$, see above \eqref{EQ: admissible specializations and Tits factoring}). In other words, we must show that the characters $\chi_{\bm y}$ take rational values (see Defn.~\ref{Defn: rational character}) on the full twist $T_{\bm \pi}$.

Consider any Galois automorphism $\sigma\in\op{Gal}\big(K(\bm y)/K(\bm x)\big)$ of the field extension. Then the Galois-conjugate character should satisfy $ ^{\sigma}(\chi_{\bm y})(T_{\bm \pi})=\zeta_{\sigma}\chi_{\bm y}(T_{\bm \pi})$ for some root of unity $\zeta_{\sigma}$, because as we see in \eqref{EQ: character values via local Coxeter} $\chi_{\bm y}(T_{\bm\pi})$ is a monomial in $\bm y$ (recall $y_{\mathcal{C}}^{N_W}=x_{\mathcal{C}}$). Now if we also call $(^{\sigma}\!\chi)$ the irreducible character of $W$ that corresponds to $^{\sigma}(\chi_{\bm y})$ via Tits' deformation theorem (but keep in mind that $(^{\sigma}\!\chi)$ is {\it not} necessarily a Galois conjugate of $\chi$), the previous equation implies $$(^{\sigma}\!\chi)(1)=\zeta_{\sigma}\chi(1),$$ which of course can only be true if $\zeta_{\sigma}=1$. The only way this can be true for {\it any} choice of $\sigma$ is if the character value was rational to begin with.
\end{proof}

\subsection{Malle's character permutations and the technical lemma}
\label{Section: Malle's character permutations and the technical lemma}

The fake degree $P_{\chi}(q):=\sum q^{e_i(\chi)}$ of an irreducible character $\chi\in\widehat{W}$ is a polynomial that records the exponents $e_i(\chi)$ of the character (see \cite[\S4.4]{lehrer-taylor-book}). Beynon and Lusztig \cite[Prop.~A]{beynon-lustzig-coxeter-numbers-genesis} had observed a remarkable reciprocity property for these polynomials. They satisfy $$P_{\chi}(q)=q^{c_{\chi}}P_{\iota(\chi)}(q^{-1}),$$where $c_{\chi}$ is the Coxeter number\footnote{However, Beynon and Lusztig, and later Malle, did not assign an epithet for these numbers; the mathematical godfathers were Gordon and Griffeth \cite{gordon-griffeth-catalan-numbers} who named them after Coxeter.} as given in Defn.~\ref{Defn: Coxeter numbers} and $\iota$ is a permutation of the irreducible characters that for Weyl groups is the identity apart from two characters of $E_7$ and four of $E_8$. 

Malle later on \cite[Thm.~6.5]{malle-fake-degrees} extended this reciprocity result for all complex reflection groups, defining a permutation of the characters $\Psi$ that is induced by a Galois action on the irreducible characters of the Hecke algebra (the two permutations satisfy $\iota(\chi)=\Psi(\chi^*)$). This permutation of Malle  is exactly the missing ingredient for the proof of Lemma~\ref{Lemma: only multiples of |g| contribute}; the characters $\chi$ for which $c_{\chi}$ is not a multiple of $|g|$ are grouped together by $\Psi$ and their contributions cancel.

\subsubsection*{A Galois action on the characters}

Recall (see \eqref{EQ: admissible specializations} and \eqref{EQ: admissible specializations and Tits factoring}) the specializations of the Hecke algebra $\mathcal{H}_{\bm x}(W)$ and $\mathcal{H}_x(W)$ that have coefficient fields $K(\bm x)$ and $K(x)$, and splitting fields $K(\bm y)$ and $K(y)$ respectively. Recall also that, after Prop.~\ref{Prop. Malle's charact. of splitting field L of H} the parameters satisfy $y_{\mathcal{C}}^{N_W}=x_{\mathcal{C}}$ and $y^{N_W}=x$.

\begin{definition}\label{Defn: Malle's character permutations}
We consider the permutations $\Psi_{\mathcal{C}}$ and $\Psi$ acting on the sets $\op{Irr}(\mathcal{H}_{\bm x}(W))$ and $\op{Irr}(\mathcal{H}_x(W))$ that are respectively induced by the Galois automorpshisms $\Sigma_{\mathcal{C}}$ (for $\mathcal{C}\in\mathcal{A}/W$) and $\Sigma$:\begin{align*}
\Sigma_{\mathcal{C}}&\in\op{Gal}\big(K(\bm y)/K(\bm x)\big) &\Sigma &\in\op{Gal}\big(K(y)/K(x)\big)    \\
y_{\mathcal{C}}&\rightarrow e^{2\pi i/N_W}\cdot y_{\mathcal{C}} & y &\rightarrow e^{2\pi i/N_W}\cdot y
\end{align*} 
In particular, they are defined via $\Psi_{\mathcal{C}}(\chi_{\bm y})(T_{\bm g}):=\Sigma_{\mathcal{C}}\big( \chi_{\bm y}(T_{\bm g})\big)$ and similarly for $\Psi$. By Tits' deformation theorem, they induce permutations on the set $\widehat{W}$ of irreducible characters of $W$, which we also denote by $\Psi_{\mathcal{C}}$ and $\Psi$.
\end{definition}

The permutations $\Psi_{\mathcal{C}}$ and $\Psi$ satisfy a set of properties with respect to the Coxeter numbers and other statistics of the characters $\chi\in\widehat{W}$:

\begin{proposition}\label{Prop: properties of Psi for Cox numbers}
For any character $\chi\in\widehat{W}$ and orbits $\mathcal{C},\mathcal{C}'\in\mathcal{A}/W$, the following are true:
\begin{tasks}[style=enumerate, after-item-skip=4mm](3)
\task $\displaystyle\Psi_{\mathcal{C}}(\chi)(1)=\chi(1)$\label{task1}
\task $\displaystyle m^{\Psi_{\mathcal{C}}(\chi)}_{\mathcal{C}',j}=m^{\chi}_{\mathcal{C}',j}$\label{task2}
\task $\displaystyle c_{\Psi_{\mathcal{C}}(\chi)_{,\mathcal{C}'}}=c_{\chi_{,\mathcal{C}'}}$\label{task3}
\end{tasks}
\end{proposition}
\begin{proof}
Since $\Psi_{\mathcal{C}}$ is induced by a Galois automorphism, it has to respect the degree of the character $\chi_{\bm y}$, hence also of $\chi$; this proves part~\ref{task1} The spectrum of any braid reflection $\bm s_{\mathcal{C}',\gamma}$ is generically rational (see Defn.~\ref{Defn: rational character}) by the defining relations \eqref{EQ: deformed order relations}. This means that the eigenvalues of any $\bm s_{\mathcal{C}',\gamma}$ in the representation that affords $\chi_{\bm y}$ live in the coefficient field $K(\bm x)$ and are therefore fixed by $\Psi_{\mathcal{C}}$. This proves part~\ref{task2} after recalling the definition of $m^{\chi}_{\mathcal{C},j}$ from the start of \S\ref{Section: Character values on roots of the full twist} and also part~\ref{task3} after Prop.~\ref{Prop: local cox num formula}. The same results are of course true for $\Psi$.
\end{proof}

The following is the key technical lemma that we have been building towards through all of Section~\ref{Section: Hecke algebras and the technical lemma}. The character calculations of Prop.~\ref{Prop: character values on roots of full twist} were included just so that the argument presented here is self-contained.

\begin{proposition}[The key technical lemma]\label{Prop: The key technical lemma}\ \newline
Let $g$ be a $\zeta$-regular element of $W$, with $\zeta=e^{2\pi il/d}$ of order $d$, $\chi\in\widehat{W}$ an irreducible character, and $\mathcal{C}\in\mathcal{A}/W$ an orbit of hyperplanes. Then, we have $$\Psi_{\mathcal{C}}(\chi)(g)=\op{exp}\big(-2\pi i \cdot \frac{lc_{\chi_{,\mathcal{C}}}}{d}\big) \cdot \chi(g) \quad\text{and}\quad\Psi(\chi)(g)=\op{exp}\big(-2\pi i \cdot \frac{lc_{\chi}}{d}\big) \cdot \chi(g).$$
\end{proposition}
\begin{proof}
By Prop.~\ref{Prop: regular elements lift to roots of the full twist}, we can lift  $g$ to some element $\bm g\in B(W)$ that is commensurable with the full twist (i.e. it satisfies $\bm g^d=\bm \pi^l$ with $(l,d)=1$). Now, replacing $x_\mathcal{C}$ with $y_{\mathcal{C}}^{N_W}$ we can rewrite the character evalueations from \eqref{EQ: character values via local Coxeter} as $$\chi_{\bm y}(T_{\bm g})=\chi(g)\cdot\prod_{\mathcal{C}'\in\mathcal{A}/W}y_{\mathcal{C}'}^{N_W(e_{\mathcal{C}'}\omega_{\mathcal{C}'}-c_{\chi,_{\mathcal{C}'}})l/d}, $$
which, after applying the Galois automorphism $\Sigma_{\mathcal{C}}$, becomes
$$ \Psi_{\mathcal{C}}(\chi_{\bm y})(T_{\bm g})= \chi(g)\cdot e^{2\pi i(e_{\mathcal{C}}\omega_{\mathcal{C}}-c_{\chi_{,\mathcal{C}}})l/d }\cdot \prod_{\mathcal{C}'\in\mathcal{A}/W}y_{\mathcal{C}'}^{N_W(e_{\mathcal{C}'}\omega_{\mathcal{C}'}-c_{\chi,_{\mathcal{C}'}})l/d}. $$ 
Now, this is really 
$$ \Psi_{\mathcal{C}}(\chi_{\bm y})(T_{\bm g})=e^{2\pi i(e_{\mathcal{C}}\omega_{\mathcal{C}}-c_{\chi_{,\mathcal{C}}})l/d }\cdot\chi_{\bm y}(T_{\bm g}),$$ which completes the proof after applying Tits' deformation theorem and recalling that $e_{\mathcal{C}}\omega_{\mathcal{C}}$ is a multiple of $d$ by Corol.~\ref{Corol: regular numbers divide e_co_c}. The same argument of course works for $\Psi$.
\end{proof}

We are now ready to prove Lemma~\ref{Lemma: only multiples of |g| contribute}. Only Malle's permutation $\Psi$ is sufficient for that, while the ``local" version $\Psi_{\mathcal{C}}$ will be used in Section~\ref{Section: The weighted enumeration} to deduce similar results for generating functions of  weighted reflection factorizations. 

\begin{proof}[Proof of Lemma~\ref{Lemma: only multiples of |g| contribute}]
We consider the partition of the set of irreducible characters $\chi\in\widehat{W}$ into orbits under the action of $\Psi$. We will show that the total contribution of the characters in any orbit that is not a singleton is $0$. 

Consider a character $\chi$ in such an orbit and let $k$ be the smallest number such that $\Psi^k(\chi)=\chi$. Notice, that after Prop.~\ref{Prop: The key technical lemma} we must have $k=\frac{d}{\op{gcd}(c_{\chi},d)}$. Since by Prop.~\ref{Prop: properties of Psi for Cox numbers} the degrees $\chi(1)$ as well as the Coxeter numbers $c_{\chi}$ are not affected by $\Psi$, it is sufficient to show that $$\sum_{j=1}^k\Psi^j(\chi)(g^{-1})=0.$$ But if $\xi=\op{exp}(-2\pi ilc_{\chi_{,\mathcal{C}}}/d)$, we have by Prop.~\ref{Prop: The key technical lemma} that $\Psi^j(\chi)(g^{-1})=\xi^j\chi(g^{-1})$ after which the above is immediate (indeed, $\xi$ is also a $k^{\text{th}}$ root of unity).
\end{proof}

\begin{remark}
Notice that Prop.~\ref{Prop: The key technical lemma} gives some insight on why in Weyl groups the orbits under $\Psi$ can have at most two elements. Indeed, every regular element $g$ will come with (at least) a pair of regular eigenvalues $e^{\pm2\pi il/d}$. Since we can lift $g$ to $d^{\text{th}}$ roots of either powers $\bm\pi^l$ and $\bm\pi^{d-l}$, the only way the proposition is valid for both lifts is if $\op{gcd}(c_{\chi_{,\mathcal{C}}},d)\leq 2$ or $\chi(g)=0$.

More generally, for a given $\chi$ and $\mathcal{C}$, Prop.~\ref{Prop: The key technical lemma} implies that if $\chi(g)\neq 0$, then $l\cdot c_{\chi_{,\mathcal{C}}}(\op{mod} d)$ is constant for all $l$ such that $\zeta=e^{2\pi i l/d}$ is a regular eigenvalue of $g$.
\end{remark}

\subsection{On the uniformity of the proofs}
\label{Section: On the uniformity of the proofs}

Our proofs rely so far mainly on two properties that are known in a case-by-case fashion; the BMR-freeness theorem and the structure of the splitting fields for the Hecke algebras. Both of those are known uniformly for real reflection groups (\cite[Thm.~4.4.6]{geck-pfeiffer-book} and \cite[Thm.~5]{opdam-remark-on-irr-characters}).

In fact, we could do away with the second reliance. Opdam's work \cite[Thm.~6.7]{opdam-fake-degrees-arxiv} is sufficient information for the structure of the group $\op{Gal}\big(\CC(\bm v)/\CC(\bm u)\big)$ which in turn is all we need to define the permutations $\Psi_{\mathcal{C}}\in\op{Perm}\big( \op{Irr}(W)\big)$. In fact Opdam's elements $g_{\mathcal{C},0}$ of this Galois group correspond precisely to our $\Sigma_{\mathcal{C}}$ of Defn.~\ref{Defn: Malle's character permutations} (see [ibid,~Prop.~7.1] and the discussion before [ibid,~Prop.~7.4]). We have chosen not to follow Opdam's presentation here (which involves the KZ-connection, a much more complicated beast) even if it is more uniform, as it does not eventually illuminate Prop.~\ref{Prop: The key technical lemma} much better. 

As far as the BMR-freeness theorem goes, and again because we are really interested in the ``geometric" Galois group $\op{Gal}\big(\CC(\bm v)/\CC(\bm u)\big)$, it is possible that we could replace it by Losev's weaker but uniform theorem \cite{losev-finite-dim-quotients}. We hope to be able to clarify this in the future.

\section{The weighted enumeration}
\label{Section: The weighted enumeration}

The following section studies the weighted enumeration of reflection factorizations as considered in \cite{reiner-delMas-Hameister}, where each reflection $t\in\mathcal{R}$ is weighted by the orbit $\mathcal{C}\in\mathcal{A}/W$ of its fixed hyperplane $V^t$. It provides a uniform proof of their result and extends it in a similar direction as with the Chapuy-Stump formula \eqref{EQ: Chapuy-Stump formula}. Again we assume that $W$ is irreducible (but see \S\ref{Section: When W is reducible}).

\begin{definition}\label{Defn: weight function}
Consider a set of variables $\bm w:=(w_{\mathcal{C}})_{(\mathcal{C}\in\mathcal{A}/W)}$ and a weight function $$\op{wt}:\mathcal{R}\rightarrow\{ w_{\mathcal{C}}\ |\ \mathcal{C}\in\mathcal{A}/W\},$$ such that $\op{wt}(t)=w_{\mathcal{C}}$ if $\mathcal{C}$ is the orbit that contains the fixed hyperplane $V^t$. Then, the weighted enumeration of reflection factorizations of some element $g\in W$ is encoded via the following generating function:
$$\op{FAC}_{W,g}(\bm w,z):=\sum_{\substack{(t_1,\cdots,t_N)\in\mathcal{R}^N \\t_1\cdots t_N=g}}\op{wt}(t_1)\cdots\op{wt}(t_N)\cdot \dfrac{z^N}{N!}. $$
\end{definition}

Because the sets $\mathcal{C}^{\op{ref}}:=\{t\in\mathcal{R}\ |\ V^t\in\mathcal{C}\}$ are closed under conjugation, the Lemma of Frobenius can again be used to express $\op{FAC}_{W,g}(\bm w,z)$ as a finite sum of exponentials. Notice first, that the order of the subsets $A_i$ in Thm.~\ref{Thm: Frobenius lemma} does not affect the enumeration as the different sets of factorizations have the same size. Indeed, one can easily construct a bijective map by considering a sequence of {\it Hurwitz moves}:$$
(t_1,t_2,\cdots,t_k,t_{k+1},\cdots ,t_l)\rightarrow (t_1,t_2,\cdots,t_kt_{k+1}t_k^{-1},t_k,\cdots,t_l).$$  

Having said that, and assuming there are $r=|\mathcal{A}/W|$ different orbits of hyperplanes, denoted $\mathcal{C}_1,\cdots, \mathcal{C}_r$, Thm.~\ref{Thm: Frobenius lemma} now implies that \begin{align*}
\op{FAC}_{W,g}(\bm w,z)=&\sum_{\substack{ N\geq 0\\l_1+\cdots +l_r=N}}\binom{N}{l_1,\cdots,l_r}\times\\
&\times\dfrac{1}{|W|}\sum_{\chi\in\widehat{W}}\chi(1)\cdot \chi(g^{-1})\cdot \Big[ \dfrac{\chi(\mathcal{C}^{\op{ref}}_1)}{\chi(1)}\Big]^{l_1}\cdots \Big[ \dfrac{\chi(\mathcal{C}^{\op{ref}}_r)}{\chi(1)}\Big]^{l_r}\cdot w_{\mathcal{C}_1}^{l_1}\cdots w_{\mathcal{C}_r}^{l_r} \dfrac{z^N}{N!}.\end{align*}Using standard properties of exponential generating functions, we can rewrite the sum as$$ 
\op{FAC}_{W,g}(\bm w,z)=\dfrac{1}{|W|}\sum_{\chi\in\widehat{W}}\chi(1)\cdot\chi(g^{-1})\cdot\op{exp}\Big[ zw_{\mathcal{C}_1}^{\phantom{l_1}}\cdot \dfrac{\chi(\mathcal{C}_1^{\op{ref}})}{\chi(1)}\Big]\cdots \op{exp}\Big[ zw_{\mathcal{C}_r}^{\phantom{l_r}}\cdot \dfrac{\chi(\mathcal{C}_1^{\op{ref}})}{\chi(1)}\Big].$$
Finally, notice that by Defn.~\ref{Defn: Local Coxeter numbers} we can rewrite the quantities in the exponentials in terms of local Coxeter numbers. Indeed, we have $c_{\chi_{,\mathcal{C}}}=|\mathcal{C}^{\op{ref}}|-\chi(\mathcal{C}^{\op{ref}})/\chi(1)$ and if we define $\op{wt}(\mathcal{R}):=\sum_{t\in\mathcal{R}}\op{wt}(t)$, the previous expression becomes a direct analog of \eqref{EQ: Frobenius via Coxeter numbers} :\begin{equation}
\op{FAC}_{W,g}(\bm w,z)=\dfrac{e^{z\cdot\op{wt}(\mathcal{R})}}{|W|}\sum_{\chi\in\widehat{W}}\chi(1)\cdot\chi(g^{-1})\cdot \big(e^{-zw_{\mathcal{C}_1}^{\phantom{l_1}}}\big)^{\displaystyle c_{\chi_{,\mathcal{C}_1}}}\cdots \big( e^{-zw_{\mathcal{C}_r}^{\phantom{l_r}}}\big)^{\displaystyle c_{\chi_{,\mathcal{C}_r}}}.\label{EQ: weighted enum via local Coxeter numbers}
\end{equation}

\begin{lemma}\label{Lemma: weigted enumeration only multiples of |g| contribute}
For a complex reflection group $W$, and a \textbf{regular} element $g\in W$, the total contribution in \eqref{EQ: weighted enum via local Coxeter numbers} of those characters $\chi\in\widehat{W}$ for which \textbf{any} $c_{\chi_{,\mathcal{C}}}$ is not a multiple of $|g|$ is $0$.
\end{lemma}
\begin{proof}
The proof is essentially the same as for Lemma~\ref{Lemma: only multiples of |g| contribute}. However, we first need to order the orbits $\mathcal{C}\in\mathcal{A}/W$ (arbitrarily) and then apply the same idea sequentially.

We start by partitioning the set of irreducible characters $\chi\in\widehat{W}$ into orbits under the action of $\Psi_{\mathcal{C}_1}$. Pick a character $\chi$ whose orbit is {\it not} a singleton and let $k$ be the smallest number such that $\Psi_{\mathcal{C}_1}^k(\chi)=\chi$ (again, we will have $k=\frac{|g|}{\op{gcd}(c_{\chi_{,\mathcal{C}_1}},|g|)}).$ Now, since by Prop.~\ref{Prop: properties of Psi for Cox numbers} the degrees of characters and the (local) Coxeter numbers are respected by $\Psi_{\mathcal{C}_1}$, it is enough to show that $$\sum_{j=1}^k\Psi^j_{\mathcal{C}_1}(\chi)(g^{-1})=0.$$ Indeed, this follows immediately from Prop.~\ref{Prop: The key technical lemma} as $\Psi_{\mathcal{C}_1}^j(\chi)(g^{-1})=\xi^j\chi(g^{-1})$ for some $k^{\text{th}}$ root of unity $\xi$. Notice now that we can continue with the {\it remaining} characters and the orbit $\mathcal{C}_2$ without worrying that we might eventually cancel the same character twice.
\end{proof}

Before we proceed with our structural result for weighted enumeration formulas, we introduce the following combinatorial generalizations of the length function $l_R(g)$:

\begin{definition}\label{Defn: special reflection lengths n_C}
For an arbitrary element $g\in W$ and an orbit $\mathcal{C}\in\mathcal{A}/W$, we define $n_{\mathcal{C}}(g)$ to be the smallest number of reflections in $\mathcal{C}^{\op{ref}}$ that may appear in {\it any} reflection factorization of $g$ (i.e. not necessarily reduced).
\end{definition}

\begin{remark}
Notice that it is not always true that $\sum n_{\mathcal{C}}(g)=l_R(g)$. Indeed, the element $g:=(12\bar{1}\bar{2})=-\bm 1$ in $B_2$ (which is the square of the Coxeter element) can be written both as $g=(12)(1\bar{2})$ and as $g=(1\bar{1})(2\bar{2})$, so that $n_1(g)=n_2(g)=0$.
\end{remark}

\begin{theorem}\label{Thm: Main, weighted}
For a complex reflection group $W$ and a regular element $g\in W$, the exponential generating function $\op{FAC}_{W,g}(\bm w,z)$ of weighted reflection factorizations of $g$ takes the form:$$
\op{FAC}_{W,g}(\bm w,z)=\dfrac{e^{z\cdot\op{wt}(\mathcal{R})}}{|W|}\cdot \Big[ \Phi({\bm X})\cdot\prod_{\mathcal{C}\in\mathcal{A}/W} (1-X_\mathcal{C})^{n^{\phantom{l}}_{\mathcal{C}}(g)}\Big]\Big|_{X_{\mathcal{C}}=e^{-zw_{\mathcal{C}}|g|}}.$$ Here, $\Phi({\bm X})$ is a polynomial of degree $(e_{\mathcal{C}}\cdot \omega_{\mathcal{C}})/|g|-n_{\mathcal{C}}(g)$ on each of its variables $X_{\mathcal{C}}$, it has constant term $\Phi(\bm 0)=1$, and it is not further divisible by $(1-X_{\mathcal{C}})$ for any $X_{\mathcal{C}}$. The exponents satisfy $$\dfrac{e_{\mathcal{C}}\omega_{\mathcal{C}}}{|g|}\ \geq\  n_{\mathcal{C}}(g)\ \geq\  l_R(g)-\dfrac{|\mathcal{R}|+|\mathcal{A}|-e_{\mathcal{C}}\omega_{\mathcal{C}}}{|g|}.$$
\end{theorem}
\begin{proof}
The proof is very similar to that of Thm.~\ref{Thm: Main, not weighted}. After Lemma~\ref{Lemma: weigted enumeration only multiples of |g| contribute}, we need only consider in \eqref{EQ: weighted enum via local Coxeter numbers} those characters $\chi$ for which all $c_{\chi_{,\mathcal{C}}}$ are multiples of $|g|$. This allows us to write the exponential function as $$\op{FAC}_{W,g}(\bm w,z)=\dfrac{e^{z\cdot\op{wt}(\mathcal{R})}}{|W|}\cdot\tilde{\Phi}(\bm X),$$ for a polynomial $\tilde{\Phi}$ on variables $\bm X:=(X_{\mathcal{C}})_{\mathcal{C}\in\mathcal{A}/W}$, by setting $X_{\mathcal{C}}=\big( e^{-zw_{\mathcal{C}}^{\phantom{l_r}}}\big)^{|g|}$. By Corol.~\ref{Corol: local cox num are integers andbounds} the polynomial $\tilde{\Phi}(\bm X)$ has degree $(e_{\mathcal{C}}\omega_{\mathcal{C}})/|g|$ on each of its variables $X_{\mathcal{C}}$, and it has constant term $1$ since all $c_{\chi_{,\mathcal{C}}}$ can be {\it simultaneously} $0$ only for the trivial representation.

To find the largest power of $(1-X_{\mathcal{C}})$ that divides $\tilde{\Phi}(\bm X)$, we view $\tilde{\Phi}$ as a polynomial in the single variable $X_{\mathcal{C}}$ and treat the other $X_{\mathcal{C}'}$'s as complex scalars. This is equivalent to assigning arbitrary values on all variables $w_{\mathcal{C}'}\neq w_{\mathcal{C}}$ of the weight function in Defn.~\ref{Defn: weight function}. If we further fix $z=1$, the enumerative intepretation of $(e^{\op{wt}(\mathcal{R})}/|W|)\cdot\tilde{\Phi}(X_{\mathcal{C}})$ is then that it counts weighted reflection factorizations of $g$ keeping track only of the number of reflections that fix a hyperplane in $\mathcal{C}$.

Now, as in Thm.~\ref{Thm: Main, not weighted} consider the root factorization of $\tilde{\Phi}(X_{\mathcal{C}})$: $$\tilde{\Phi}(X_{\mathcal{C}})=a(\alpha_1-X_\mathcal{C})(\alpha_2-X_{\mathcal{C}})\cdots (\alpha_r-X_{\mathcal{C}}),$$ with $r=(e_{\mathcal{C}}\omega_{\mathcal{C}})/|g|$. We see again that by plugging back $X_{\mathcal{C}}=e^{-w_{\mathcal{C}}|g|}$ each root contributes a factor of either $(\alpha_i-1)$ or $w_{\mathcal{C}}|g|$ to the leading term of the generating function. Since by Defn.~\ref{Defn: special reflection lengths n_C} this must be a scalar multiple of $w_{\mathcal{C}}^{n_C(g)}$, we have that $(1-X_{\mathcal{C}})^{n_{\mathcal{C}}(g)}$ divides $\tilde{\Phi}(X_{\mathcal{C}})$ and is the largest power that does so (this furthermore proves the first inequality). Since this is true for a dense set of the complex values $X_{\mathcal{C}'}$, we in fact have that $(1-X_{\mathcal{C}})^{n_{\mathcal{C}}(g)}$ is a maximal factor of $\tilde{\Phi}(\bm X)$.

The only thing left to show is the second inequality for the $n_{\mathcal{C}}(g)$'s. To see this, we now identify all weights $w_{\mathcal{C}'},\ \mathcal{C}'\neq\mathcal{C}$ to a single weight $w$, set again $z=1$, and treat $\tilde{\Phi}$ as a polynomial on two variables $X=e^{-w|g|}$ and $X_{\mathcal{C}}=e^{-w^{\phantom{l}}_{\mathcal{C}}|g|}$. The general argument about $\tilde{\Phi}(\bm X)$ implies that we can consider the polynomial $\Phi'(X,X_{\mathcal{C}})$ defined by $$\Phi'(X,X_{\mathcal{C}}):=\dfrac{\tilde{\Phi}(X,X_{\mathcal{C}})}{(1-X_{\mathcal{C}})^{n_{\mathcal{C}}(g)}}.$$
Now, the generating function $$\dfrac{e^{\op{wt}(\mathcal{R})}}{|W|}\cdot\Phi'(X,X_{\mathcal{C}})\cdot(1-X_{\mathcal{C}})^{n_{\mathcal{C}}(g)}$$ counts 
reflection factorizations of $g$ weighing reflections in $\mathcal{C}^{\op{ref}}$ by $w_{\mathcal{C}}$ and the rest by $w$. We want to enumerate factorizations that have exactly the minimal number $n_{\mathcal{C}}(g)$ of reflections of type $\mathcal{C}$. Since the term $(1-X_{\mathcal{C}})^{n_{\mathcal{C}}(g)}$ always contributes a factor of $(w^{\phantom{l}}_{\mathcal{C}}|g|)^{n_{\mathcal{C}}(g)}$ to the Taylor expansion, the answer to the previous question would be given by $$\dfrac{|g|^{n_{\mathcal{C}}(g)}}{|W|}\times e^{\op{wt}(\mathcal{R})}\Big|_{w_{\mathcal{C}}=0}\times\Phi'(X,X_{\mathcal{C}})\Bigg|_{\substack{X_{\mathcal{C}}=1\quad\  \\ X=e^{-w|g|}} }.$$
The leading term of this exponential generating function should clearly be a multiple of $w^{l_R(g)-n_{\mathcal{C}}(g)}$. As in the previous argument, this implies that $\Phi'(X,1)$ is a multiple of $(1-X)^{l_R-n_{\mathcal{C}}(g)}$, but since by construction its degree is equal to $\sum_{\mathcal{C}'\neq\mathcal{C}}e_{\mathcal{C}'}\omega_{\mathcal{C}'}=|\mathcal{R}|+|\mathcal{A}|-e_{\mathcal{C}}\omega_{\mathcal{C}}$, we must have $$l_R(g)-n_{\mathcal{C}}(g)\leq |\mathcal{R}|+|\mathcal{A}|-e_{\mathcal{C}}\omega_{\mathcal{C}},$$ which completes the proof.
\end{proof}

\begin{corollary}
For a complex reflection group $W$ and a regular element $g\in W$ of order $|g|=d_n$, the weighted reflection factorizations of $g$ are counted by the formula:
$$\op{FAC}_{W,g}(\bm w,z)=\dfrac{e^{z\cdot\op{wt}(\mathcal{R})}}{|W|}\cdot\prod_{\mathcal{C}\in\mathcal{A}/W}\big(1-e^{-zw_{\mathcal{C}}^{\phantom{l_r}}\!|g|}\big)^{n_\mathcal{C}(g)},$$
where the exponents are explicitly given by $n_{\mathcal{C}}(g)=(e_{\mathcal{C}}\omega_{\mathcal{C}})/|g|$.
\end{corollary}
\begin{proof}
As we showed in the proof of Corol.~\ref{Cor: Chapuy-Stump for dn}, when $g$ is some $d_n$-regular element we must have $l_R(g)=(|\mathcal{R}|+|\mathcal{A}|)/|g|$. Then the previous theorem implies that $n_{\mathcal{C}}(g)=(e_{\mathcal{C}}\omega_{\mathcal{C}})/|g|$, which further forces the equality $\Phi(\bm X)=1$ and hence completes the argument.
\end{proof}

\begin{remark}
For well-generated groups $W$, we always have $|c|=d_n$ so that the previous Corollary recovers the main theorem of \cite{reiner-delMas-Hameister} and extends it to the groups of Prop.~\ref{Prop: When d_n is regular}. Notice that while in well-generated groups we have at most two orbits of hyperplanes, the groups $G_{7},G_{11},G_{15},G_{19}$ have three orbits. For all of them but $G_{15}$, $d_n$ is regular.
\end{remark}

\subsection{When $W$ is reducible}
\label{Section: When W is reducible}

So far to simplify the arguments, we have silently assumed everywhere that $W$ is irreducible. This is not a real restriction though and in fact the statement of Thm.~\ref{Thm: Main, weighted} remains true essentially as is.

Indeed, assume that $W=W_1\times\cdots\times W_k$ acts on the space $V=V_1\oplus \cdots\oplus V_k$, with $W_i$ acting irreducibly on $V_i$. Then, a regular eigenvector $\bm v=(v_1,\cdots,v_k)$ must have all $v_i$'s regular in their respective groups too and hence a regular element $W\ni g=g_1\cdots g_k$ must have all $g_i$'s regular in the $W_i$'s. Moreover since reflections from different $W_i$'s commute, the corresponding weighted generating function would just be the product $$\op{FAC}_{W,g}(\bm w,z)=\prod_{i=1}^k\op{FAC}_{W_i,g_i}(\bm w,z).$$
Since the hyperplane orbits $\mathcal{C}\in\mathcal{A}/W$ are the disjoint union of the orbits $\mathcal{C}'\in\mathcal{A}_i/W_i$ the statement of Thm.~\ref{Thm: Main, weighted} remains valid if we only change the evaluation of $X_{\mathcal{C}}$ from $e^{-zw^{\phantom{l}}_{\mathcal{C}}|g|}$ to $e^{-zw^{\phantom{l}}_{\mathcal{C}}|g_i|}$, where $g_i$ is the regular element in the group $W_i$ that contains the orbit $\mathcal{C}$.

\section*{Acknowledgments} 
We would like to thank Jean Michel and Maria Chlouveraki for interesting discussions and their guidance {\`a} propos the Hecke algebras. We would also like to thank Guillaume Chapuy for his detailed explanation of some calculations in \cite{chapuy-stump} and 
Christian Stump for showing us how to use SAGE \cite{sagemath} to derive formulas like \eqref{Eq: Frob formula no Cox}. It was at this vantage point that the existence of a statement like Lemma~\ref{Lemma: only multiples of |g| contribute} became clear.

We would also like to thank Ivan Marin for his detailed suggestions vis-{\`a}-vis the uniformity of the proofs and Losev's theorem.

\printbibliography

\Address

\end{document}